\documentclass[article]{amsart}
\usepackage[T1]{fontenc}
\usepackage{amsthm}
\usepackage{amssymb}
\usepackage{amsmath}
\usepackage{mathabx}

\usepackage{xcolor}
\usepackage{verbatim, color}

\usepackage{bbm}
\usepackage{float}
\usepackage{dsfont, graphicx}
\usepackage{epstopdf}
\usepackage[hidelinks]{hyperref}
\usepackage[T1]{fontenc}
\usepackage{amssymb,amsmath, amsthm, amsfonts}
\usepackage{graphicx}
\usepackage{tikz}
\usepackage{tkz-euclide}

\usepackage[hidelinks]{hyperref}
\usepackage{pagerange}
\usepackage{thmtools,thm-restate}

\usepackage{tikz}
\usepackage{tikz-3dplot}
\usetikzlibrary{3d}

\makeatletter
\numberwithin{equation}{section}
\numberwithin{figure}{section}
\theoremstyle{plain}
\newtheorem{thm}{Theorem}[section]
\newtheorem{prop}[thm]{Proposition}
\newtheorem{definition}[thm]{Definition}

\newtheorem{lem}[thm]{Lemma}
\newtheorem{cor}[thm]{Corollary}
\newtheorem{rem}[thm]{Remark}

  \newcounter{casectr}
  
\theoremstyle{definition}

\theoremstyle{remark}

\newcommand{\R}{\mathbb{R}}

\def \bR {\Bbb R}
\def \bN {\Bbb N}
\def \bZ {\Bbb Z}

\def \bT {\Bbb T}

\def \l {\left}
\def \r {\right}

\def \lea {\lesssim}
\def \leaa {\lessapprox}

\def \gea {\gtrsim}

\def \w {\widetilde}
\def \p {^\prime}

\def \ed {e^{it\Delta}}
\def \ca {\mathbf{1}}

\begin{document}

\address{Yangkendi Deng
\newline \indent Department of Mathematics and Statistics, Beijing Institute of Technology.
\newline \indent Key Laboratory of Algebraic Lie Theory and Analysis of Ministry of Education.
\newline \indent  Beijing, China. \indent}
\email{dengyangkendi@bit.edu.cn}

\address{Boning Di
\newline \indent School of Mathematical Sciences, Beijing University of Posts and Telecommunications, Beijing, China
\newline \indent Key Laboratory of Mathematics and Information Networks (Beijing University of Posts and Telecommunications), Ministry of Education, Beijing, China}
\email{diboning@bupt.edu.cn}

\address{Jiao Ma
\newline \indent School of Mathematics and Physics, University of Science and Technology Beijing, Beijing, China. }
\email{majiao191@mails.ucas.ac.cn}

\address{Dunyan Yan
\newline \indent School of Mathematical Sciences, University of Chinese Academy of Sciences, Beijing, China}
\email{ydunyan@ucas.ac.cn}

	\address{Kailong Yang
		\newline \indent National Center for Applied Mathematics in Chongqing, Chongqing Normal University, Chongqing  \\
		\indent China. }
	\email{ykailong@mail.ustc.edu.cn}

\title{Long-time Strichartz estimates on 3D waveguide with applications}
\author{Yangkendi Deng, Boning Di, Jiao Ma, Dunyan Yan, Kailong Yang}

\maketitle

\begin{abstract}
We study long-time Strichartz estimates for the  Schr\"{o}dinger equation on waveguide manifolds, and use them to establish upper bounds on the growth of Sobolev norms for the nonlinear  Schr\"{o}dinger equation on three-dimensional waveguides.

\end{abstract}



\section{Introduction}

\subsection{Background and Motivations}
Strichartz estimates are useful tools to study dispersive equations. A wide range of research concerning Strichartz estimates on a variety of manifolds has attracted attention over the past few decades, as it helps us to study the behavior of related dispersive equations on these manifolds. The classical Strichartz estimates for Schr\"{o}dinger equation on $\bR^d$ (\cite{Strichartz1977}) read as
$$\| \ed \phi \|_{L^{\frac{2(d+2)}{d}}_{t,x} (\bR\times\bR^d )}\lea \|\phi\|_{L^2(\bR^d)}.$$
For torus $\bT^d=(\bR/\bZ)^d$, it holds locally in time with an additional $\varepsilon$-derivative loss(\cite{BD2015}), it reads as
$$\| \ed P_{\le N}\phi \|_{L^{\frac{2(d+2)}{d}}_{t,x} ([0,1]\times\bT^d )}\lea _{\varepsilon} N^{\varepsilon} \|\phi\|_{L^2(\bT^d)}.$$

Waveguide manifolds, of the form $\mathbb{R}^m \times \mathbb{T}^n$, are product spaces of Euclidean space with a torus. The study on waveguide manifolds is particularly interesting due to the unique combination of properties inherited from both Euclidean spaces and tori, offering deeper insights into the underlying physics. For the parallel estimates on 2D waveguide $\bR \times \bT$, Takaoka and Tzvetkov \cite{TT2001} gave sharp estimate
$$\| \ed \phi\|_{L^{4}_{t,x} ([0,1]\times\bR\times \bT )}\lea   \|\phi\|_{L^2(\bR\times \bT)}.$$
The estimate on $\mathbb{R} \times \mathbb{T}$ is different from the one on Euclidean space (which holds for all time) and also different from the one on a pure torus (which needs a frequency cut and has a derivative loss). It is still unknown whether similar estimates—without frequency cut or derivative loss and on the unit time scale—hold for higher-dimensional waveguides like $\mathbb{R}^m \times \mathbb{T}^n$ (with $m+n\ge 3$), or even for $\mathbb{R}^{d-1} \times \mathbb{T}$ when $d\ge 3$.

A key result is the global-in-time Strichartz estimate for Schr\"odinger equations on $\mathbb{R}^m \times \mathbb{T}^n$, proved by Barron \cite{Barron2021} with an additional $\varepsilon$-derivative loss at the endpoint exponent. This estimate reflects the dispersive effect contributed specifically by the Euclidean factor of $\mathbb{R}^m \times \mathbb{T}^n$.
The main result is the following. Let $\mathbb{T}^n$ be an $n$-dimensional rational or irrational torus. Assume $m\ge 1$, $n\ge 1$, and set $d=m+n$, $p^* =2+\frac{4}{d}$. For any $p>p^*$, let $q=\frac{4p}{m(p-2)}$ and $s=\frac{d}{2}-\frac{d+2}{p}$. If $q>2$, then estimate 
\begin{equation}\label{global in time strichartz estimate RmTn}
\Bigl(\sum_{\gamma \in \mathbb{Z}}\| e^{it\Delta} \phi \|_{L_{t,x}^p([\gamma,\gamma+1] \times \mathbb{R}^m\times \mathbb{T}^n )}^q \Bigr)^{\frac{1}{q}}\lesssim \| \phi \|_{H^s(\mathbb{R}^m\times \mathbb{T}^n)},
\end{equation}
 holds. Furthermore, when $p=p^*$, \eqref{global in time strichartz estimate RmTn} holds for $q=q(p^*)$ and any $s>0$.
 In another work, Barron, Christ and Pausader \cite{BCP2021} established global-in-time Strichartz estimates for the Schr\"odinger equation on the product space $\mathbb{R} \times \mathbb{T}$, overcoming the derivative loss at the endpoint. This corresponds to the case $m=n=1, p=4, q=8, s=0$ in estimate \eqref{global in time strichartz estimate RmTn}.

This paper addresses another type of long-time Strichartz estimate. In this context, a fundamental result was established by Deng, Germain, and Guth \cite{DGG2017} on a rectangular torus
$\mathbb{T}_{\ell}^d=[0,\ell_1]\times \cdots \times [0,\ell_d],$
\begin{equation}\label{DGG}
\|e^{it\Delta} P_{\le N} \phi\|_{L^p( [0,T]\times \mathbb{T}_\ell^d)}\lesssim_{\varepsilon} N^{\varepsilon}\big(T^{\frac{1}{p}}+N^{\frac{d}{2}-\frac{d+2}{p}}+N^{\frac{d}{2}-\frac{3d}{p}}T^{\frac{1}{p}}\big)\|\phi\|_{L^2( \mathbb{T}_\ell^d)},
\end{equation}
which holds for a generic choice of $\ell=(\ell_1,\cdots,\ell_d)$. More precisely, this corresponds to requiring that $\ell$ satisfy a Diophantine condition of the form
$$ \left|   \sum_{j=1}^d \frac{n_j}{\ell_j^2}   \right| \ge C^{-1} \left( \sum_{j=1}^d |n_j| \right)^{-(d-1+\delta)}, \quad \forall (n_1,\cdots, n_d)\in \bZ^d/\{0\}, $$
where $C, \delta>0$ are uniform constants.

The tori they considered are irrational. Analogously, the waveguide manifolds studied here represent a generalized form of such tori. This paper derives \eqref{DGG}-type estimates on three-dimensional waveguide manifolds and applies them to NLS problems.

Consider the nonlinear
Schr\"odinger equations (NLS) on the waveguides
$\mathbb{R}^m \times \mathbb{T}^n$:
\begin{equation}\label{nls1}
\begin{cases}
i\partial_t u + \Delta_{\mathbb{R}^m 
\times \mathbb{T}^n}  u = |u|^{\mu-1} u,
\\
u(0) = u_0\in H^s(\mathbb{R}^m
\times \mathbb{T}^n),
\end{cases}
\end{equation}
where
$\Delta_{\mathbb{R}^m  \times \mathbb{T}^n}
$ is the Laplace-Beltrami operator on
$\mathbb{R}^m \times \mathbb{T}^n$ and
$u : \mathbb{R} 
\times \mathbb{R}^m\times \mathbb{T}^n  
\to \mathbb{C}$ 
is a complex-valued function. This type of equation arises as a model
 in the study of nonlinear optics
 (propagation of laser beams through the 
 atmosphere or in a plasma), 
 especially in nonlinear optics of 
 telecommunications
\cite{S,SL}. The equation \eqref{nls1} admits the following conserved quantities:
\begin{align*}
\text{mass: }    
&\quad   M[u(t)]  = \int_{\mathbb{R}^m 
\times \mathbb{T}^n} |u(t,x)|^2\,\mathrm{d}x,\\
\text{   energy:  }     &  \quad E[u(t)]  
= \int_{\mathbb{R}^m \times \mathbb{T}^n} 
\frac12 |\nabla u(t,x)|^2  + \frac1{\mu+1} 
|u(t,x)|^{\mu+1} \,\mathrm{d}x.
\end{align*}

There has been a wide range of research concerning 
the well-posedness theory and long-time
 behavior of solutions for \eqref{nls1} on
$\bR^m \times \bT^n$. The Euclidean case, i.e.
$n=0$, is studied and the theory, at least in the defocusing setting, is well established (See \cite{I-team1,Dodson3,Taobook} and the references therein). Moreover, we refer to \cite{Bourgain1993NLS,DPST,herr2025global,HTT1,IPT3,kwak2024global,kwak2024critical} for some works on tori. Due to the nature of such product spaces, the techniques used in Euclidean and tori settings are frequently combined and applied to waveguide problems. 

The study of scattering phenomena and the growth of higher-order Sobolev norms constitutes a core pillar of long-time dynamical analysis for the NLS. Scattering describes a regime where linear dispersive decay dominates nonlinear interactions, ensuring nonlinear solutions asymptotically converge to free linear evolution as $t \to \pm \infty$. 
While the growth of higher-order Sobolev norms for solutions to the NLS quantifies the long-time transfer of energy from low to high frequencies.  Unlike conserved norms, the higher Sobolev norms $\|u(t)\|_{H^s}$, $s>1$, may evolve non-trivially. 
The unbounded polynomial growth of $H^s$ norms is dynamically precluded in scattering scenarios and typically reserved for non-scattering geometries (e.g., $\bT^d$). We note that,  although the NLS on waveguide manifolds inherits certain dispersive features from the Euclidean component—thereby 
expected the possibility of scattering in the case of $\frac{4}{m}\leq \mu-1\leq \frac{4}{m+n-2}$ (see \cite{HP,MR3406826} for explanations), it is also partially confined to a compact manifold, namely the torus. This geometric confinement inhibits dispersion along the compact directions and, as a result, precludes scattering behavior.

 The first breakthrough in the growth of higher-order Sobolev norms for NLS was made by Bourgain \cite{bourgain1996growth}, who employed the high-low frequency decomposition method to refine this exponential growth estimate to a polynomial upper bound—specifically for cubic nonlinearities in 2, and 3 dimensions tori. After that, Staffilani
\cite{staffilani1997growth,staffilani1997quadratic} showed that the 
solutions to some type of KdV and
Schr\"odinger equations on
$\R$ and
$\R^2$ possess a polynomial bound 
on the time growth of the
$H^s$-norm,
$s>1$ by using fine multilinear estimates. 
Colliander-Keel-Staffilani-Takaoka-Tao 
established polynomial bounds in low 
Sobolev norm with
$s\in(0,1)$ for the NLS equation by proposing 
a new method using modified energy called 
the ``upside-down I-method" in
\cite{colliander2002polynomial}. Then, Sohinger in
\cite{sohinger2011bounds1,sohinger2011bounds} developed the upside-down 
I-method to obtain polynomial bounds on 
the growth of high Sobolev norms for NLS 
on
$\mathbb{T}$ and
$\R$. 
Additionally, we also refer to \cite{deng2024growth,planchon2017growth,takaoka2024growth} and the references herein for further developments in this topic.
Recently, Deng \cite{deng2019growth} and Deng-Germain \cite{deng2019growth2} were the first to apply long-time Strichartz estimates \eqref{DGG} to establish the growth of Sobolev norms for solutions to energy-critical and subcritical NLS on the three-dimensional irrational torus. Along this route,  we focus on the application of the long-time Strichartz estimate to \eqref{nls1} on 3-dimensional waveguide manifold $\bR\times \bT^2$ and with the nonlinearity corresponding to the case $3<\mu\leq5$, i.e.
\begin{equation}\label{nls}
i\partial_t u+\Delta u = |u|^{\mu-1}u,\,\,x\in\bR\times \bT^2, \,\, 3<\mu\leq5.
\end{equation} 

\subsection{Statement of main results}

Our goal of this paper is to study the long-time Strichartz estimate on waveguide
\begin{equation}\label{eq:long time strichartz estimate}
 \|\ed P_{\le N} \phi\|_{L_{t,x}^p([0,T] \times\bR^m\times \bT^n )}\lea C(p,T,N)\|\phi\|_{L^2(\bR^m\times \bT^n)},
\end{equation}
where $p\ge 2, T\ge 1, N\in 2^\bN, m\ge 1, n\ge 1, d=m+n$.

 Our starting point is the following estimate, derived from the global-in-time Strichartz estimates in \cite{Barron2021}:

\begin{thm}\label{thm:BCP RmTn}
Assume that $p> 2, T\ge 1, N\in 2^\bN, m\ge 1, n\ge 1, d=m+n$. For $\varepsilon>0$, there holds
\begin{equation}\label{eq:long time strichartz estimate RmTn}
 \|\ed P_{\le N} \phi\|_{L_{t,x}^p([0,T] \times\bR^m\times \bT^n )}\lea C_0(p,T,N)\|\phi\|_{L^2(\bR^m\times \bT^n)},
\end{equation}
where
\begin{equation*}
  C_0(p,T,N)= \begin{cases}
             N^\varepsilon T^{\frac{m+2}{2p}-\frac{m}{4}} , & \mbox{if~} 2< p \le 2+\frac{4}{d},  \\
             T^{\frac{m+2}{2p}-\frac{m}{4}} N^{\frac{d}{2}-\frac{d+2}{p}} , & \mbox{if~} 2+\frac{4}{d}<p <2+\frac{4}{m},  \\
              N^{\frac{d}{2}-\frac{d+2}{p}} , & \mbox{if~}  p\ge 2+\frac{4}{m}.
            \end{cases}
\end{equation*}

\end{thm}
Theorem \ref{thm:BCP RmTn} follows from the global-in-time Strichartz estimate in \cite{Barron2021} together with the H\"older's inequality. Furthermore, lower bounds for the optimal constant are obtained from several explicit examples of initial data.

\begin{thm}\label{thm:example}
Assume that $p\ge 2, T\ge 1, N\in 2^\bN, m\ge 1, n\ge 1, d=m+n$. Then there exist initial data $\phi_1, \phi_2, \phi_3 \in L^2(\bR^m\times \bT^n)$ such that
$$  \|\ed P_{\le N} \phi_1\|_{L^p([0,T] \times\bR^m\times \bT^n )}\gea T^{\frac{m+2}{2p}-\frac{m}{4}} \|\phi_1\|_{L^2(\bR^m\times \bT^n)}, $$
$$  \|\ed P_{\le N} \phi_2\|_{L^p([0,T] \times\bR^m\times \bT^n )}\gea T^{\frac{m+2}{2p}-\frac{m}{4}}N^{\frac{n}{2}-\frac{n+2}{p}} \|\phi_2\|_{L^2(\bR^m\times \bT^n)}, $$
and
$$ \|\ed P_{\le N} \phi_3\|_{L^p([0,T] \times\bR^m\times \bT^n )}\gea N^{\frac{d}{2}-\frac{d+2}{p}} \|\phi_3\|_{L^2(\bR^m\times \bT^n)}.   $$
    
\end{thm}

According to the examples in Theorem \ref{thm:example}, we guess that the sharp constant (up to $N^\varepsilon$-loss) in \eqref{eq:long time strichartz estimate}
is
\begin{equation} \label{eq:sharp constant RmTn}
  C(p,T,N)=
             T^{\frac{m+2}{2p}-\frac{m}{4}}+ T^{\frac{m+2}{2p}-\frac{m}{4}}N^{\frac{n}{2}-\frac{n+2}{p}}+  N^{\frac{d}{2}-\frac{d+2}{p}}, 
\end{equation}
for $p\in [2, \infty]$.
A comparison between the conjectured constant above and Theorem \ref{thm:BCP RmTn} shows that, apart from the case where $2+\frac4d < p < 2+\frac4m$, the estimates obtained are already sharp (up to $N^\varepsilon$-loss). 

Our first main result demonstrates that the correctness of the conjectured constant for $p=4$ can be verified. Apart from the cases already covered by Theorem \ref{thm:BCP RmTn}, we primarily consider the cases where $m=1$ and $n \ge 2$.

\begin{thm}\label{thm:long-time Strichartz for p=4}
Assume that $T\ge 1, N \in 2^\bN, n\ge 2$. Then there holds
$$ \|\ed P_{\le N} \phi\|_{L_{t,x}^4([0,T] \times\bR\times \bT^n )}\leaa (T^{\frac18}N^{\frac{n-2}{4}}+  N^{\frac{n-1}{4}} )\|\phi\|_{L^2(\bR\times \bT^n)}.  $$
\end{thm}

Here we use the notation $\leaa$ to denote $\lesssim_{\varepsilon} N^{\varepsilon}$. $L^2 \to L^4$ Strichartz estimates for the Schr\"odinger equation have been extensively studied on Euclidean spaces, on tori, and on waveguide manifolds; see, e.g., \cite{Bourgain1993NLS, DPST, DFZ2025, HK2024, KV2016, LZ2025, TT2001}.
We primarily employ classical techniques to prove Theorem \ref{thm:long-time Strichartz for p=4}. It is worth noting that recently developed methods may potentially reduce the $N^\varepsilon$ loss in Theorem \ref{thm:long-time Strichartz for p=4} to a logarithmic loss in $N$ or even to a bounded constant.

Building upon the foundation of Theorem \ref{thm:BCP RmTn}, we consider establishing improved long-time Strichartz estimates on 3D waveguide manifolds. In particular, we only consider the cases $(m,n)\in \{(2,1), (1,2)\}$.

\begin{thm}\label{thm:improved long-time Strichartz estimate}
Assume that $T\ge 1, N\in 2^\bN$. Then on $\bR^2\times \bT $, there holds for $\frac{10}{3}<p<4$,
\begin{equation}\label{eq:improved long time strichartz estimate R2T}
 \|\ed P_{\le N} \phi\|_{L_{t,x}^p([0,T] \times\bR^2\times \bT )}\leaa C_1(p,T,N)\|\phi\|_{L^2(\bR^2\times \bT)},
\end{equation}
where
\begin{equation*}
  C_1(p,T,N)= \begin{cases}
              N^{\frac{3}{2}-\frac{5}{p}} , & \mbox{if~}  T\le N^{\frac{9p}{4}-\frac{15}{2}},
              \\
               T^{\frac{1}{3p}} N^{\frac{3}{4}-\frac{5}{2p}} , & \mbox{if~} N^{\frac{9p}{4}-\frac{15}{2}}<T \le N^{\frac{3}{2}},  \\
               T^{\frac{2}{p}-\frac{1}{2}}N^{\frac{3}{2}-\frac{5}{p}} , & \mbox{if~}   T> N^{\frac{3}{2}}.
            \end{cases}
\end{equation*}
On $\bR\times \bT^2$, there holds for $\frac{10}{3}<p<4$,
\begin{equation}\label{eq:improved long time strichartz estimate RT2 l}
 \|\ed P_{\le N} \phi\|_{L_{t,x}^p([0,T] \times\bR\times \bT^2 )}\leaa C_2(p,T,N)\|\phi\|_{L^2(\bR\times \bT^2)},
\end{equation}
and for $4<p<6$,
\begin{equation}\label{eq:improved long time strichartz estimate RT2 g}
 \|\ed P_{\le N} \phi\|_{L_{t,x}^p([0,T] \times\bR\times \bT^2 )}\leaa C_3(p,T,N)\|\phi\|_{L^2(\bR\times \bT^2)},
\end{equation}
where
\begin{equation*}
   C_2(p,T,N)= \begin{cases}
              N^{\frac{3}{2}-\frac{5}{p}} , & \mbox{if~}  T\le N^{\frac{3p}{4}-\frac{5}{2}},
              \\
               T^{\frac{2}{3p}} N^{1-\frac{10}{3p}} , & \mbox{if~} N^{\frac{3p}{4}-\frac{5}{2}}<T \le N^{\frac{1}{2}},  \\
               T^{\frac{4}{p}-1}N^{\frac{3}{2}-\frac{5}{p}} , & \mbox{if~}  N^{\frac{1}{2}}< T\le N^{2},
                \\
               T^{\frac{3}{2p}-\frac{1}{4}} , & \mbox{if~}   T> N^{2}.
            \end{cases}
\end{equation*}
\begin{equation*}
   C_3(p,T,N)= \begin{cases}
              N^{\frac{3}{2}-\frac{5}{p}} , & \mbox{if~}  T\le N^{p-2},
              \\
               T^{\frac{1}{2p}} N^{1-\frac{4}{p}} , & \mbox{if~} N^{p-2}<T \le N^{4},  \\
               T^{\frac{3}{2p}-\frac{1}{4}}N^{2-\frac{8}{p}} , & \mbox{if~}  T>N^{4}.
            \end{cases}
\end{equation*}
\end{thm}
From Theorems \ref{thm:BCP RmTn}, \ref{thm:long-time Strichartz for p=4} and \ref{thm:improved long-time Strichartz estimate}, we immediately derive the following corollary. 
\begin{cor}\label{cor-long-time}  It holds that
\begin{equation}
\left\|\ed P_{\le N} \phi\right\|_{L_{t, x}^p\left(\left[0, N^{c(p)}\right] \times \bR\times \bT^2\right)} \leaa N^{\frac{3}{2}-\frac{5}{q}}\|\phi\|_{L^2(\bR\times \bT^2)},
\end{equation}where 
\begin{equation*}
    c(p)= \begin{cases}\frac{3 p-10}{4}, & \frac{10}{3}<p\leq4,\\p-2, & 4<p< 6, \\ \infty, & p \geq 6,\end{cases}
\end{equation*}
and
\begin{equation}
\left\|\ed P_{\le N} \phi\right\|_{L_{t, x}^p\left(\left[0, N^{\tilde{c}(p)}\right] \times \bR^2\times \bT\right)} \leaa N^{\frac{3}{2}-\frac{5}{q}}\|\phi\|_{L^2(\bR^2\times \bT)},
\end{equation}where 
\begin{equation*}
   \tilde{c}(p)= \begin{cases}\frac{3(3 p-10)}{4}, & \frac{10}{3}<p<4, \\ \infty, & p \geq 4.\end{cases}
\end{equation*}
\end{cor}

Finally, we discuss the application of Theorem \ref{thm:improved long-time Strichartz estimate}.  We focus on $\mathbb{R} \times \mathbb{T}^2$, as the $\mathbb{R}^2 \times \mathbb{T}$ case already has stronger scattering results \cite{HP,tzvetkov2016well}. Our result about the growth of higher Sobolev norms is stated as follows.

\begin{thm}[Growth of higher Sobolev norms to NLS on 3D waveguide]\label{thm:growth of solutions to NLS on 3D waveguide} Suppose $u$ is a solution to \eqref{nls} with finite energy and $\|u(0)\|_{H^s\left(\bR\times \bT^2\right)}=A$.
 Define the exponent $\omega(s,\mu)$ for $3<\mu\leq5$ by
 \begin{equation*}
		\omega(s,\mu) :=
		\begin{cases}
			\frac{s}{5-\mu+\theta(\mu)}, \ 3<\mu<5,\\
			300(s-1), \ \mu=5,
		\end{cases}
	\end{equation*}
where $\theta(\mu)=\frac{|(\mu-3)(5-\mu)|}{2(65\mu-162)}\cdot\frac{6-\mu}{16-\mu}>0$  for $3<\mu<5$.
Then for $s>1$, the estimate
 $$
\|u(t)\|_{H^s\left(\bR\times \bT^2\right)} \lesssim A+(1+|t|)^{\omega(s,\mu)}
$$ holds for all  $3<\mu<5$, and for $\mu=5$ under the  additional small-energy assumption that $E[u] \leq \eta^2$ for some constant $\eta>0$.
\end{thm}
\begin{rem}

For the case $\mu=5$, scattering for $H^1$ initial data has already been established in \cite{HP}, which implies that high-order Sobolev norms are globally bounded rather than growing polynomially. Nevertheless, to illustrate the applicability of Theorem \ref{thm:improved long-time Strichartz estimate}, we derive the polynomial growth bound for small energy solutions.
\end{rem}

\subsection{Organization of the rest of this paper}
In Section $2$, we primarily present the long-time Strichartz estimates (Theorem \ref{thm:BCP RmTn}) derived from the main results in \cite{Barron2021}, along with some related counterexamples (Theorem \ref{thm:example}). In Section $3$, we present the main analytical results, namely the sharp long-time Strichartz estimates for $p=4$ (Theorem \ref{thm:long-time Strichartz for p=4}) and the improved long-time Strichartz estimates (Theorem \ref{thm:improved long-time Strichartz estimate}). In Section $4$, we discuss relevant applications of Theorem \ref{thm:improved long-time Strichartz estimate} to the NLS and obtain a  polynomial upper bound on the growth of the $H^s$-norm of the solution to \eqref{nls} on $\bR \times \bT^2$.

\subsection{Notation} In this subsection, we summarize some notations that would be used throughout this paper. For nonnegative quantities $A$ and $B$, we write $A\lesssim B$ if there exists a constant $C>0$ such that $A\leq C B$. And we use the notation $\leaa$ to denote $\lesssim_{\varepsilon} N^{\varepsilon}$. The notation $O(1)$ denotes a bounded constant whose value may change from line to line. We further write $o(1)$ for a quantity tending to zero, and use the shorthand $a+$ (respectively, $a-$) to denote any quantity greater (respectively, smaller) than $a$. 

We use $\chi$ to denote a Schwartz function such that $\ca_{[-1,1]}\le \chi \le \ca_{[-2,2]}$. 

For simplicity, we denote
$$ \int_{\bR^m\times \bZ^n} F(\xi) {\rm d} \xi := \int_{\bR^m}\sum_{\xi_2\in \bZ^n} F(\xi_1, \xi_2) {\rm d} \xi_1$$
and
$$ \|F\|_{L^2(\bR^m\times \bZ^n)}^2:=\int_{\bR^m}\sum_{\xi_2\in \bZ^n} |F(\xi_1, \xi_2)|^2 {\rm d} \xi_1.  $$

For each dyadic number $N\geq 1$, we define the Fourier multiplier operator $\mathcal{D}_N$ by
\begin{equation*}
\widehat{\mathcal{D}_N u}(\xi)=g(\xi/N)\widehat{u}(\xi),
\end{equation*}
where $g(y)=g(|y|)$ is a smooth radial cutoff satisfying
\[
g(y)=1 \quad \text{for } |y|\leq 1,
\qquad 
g(y)=|y|^{\,s-1} \quad \text{for } |y|\geq 2.
\]
Moreover, once $N$ is fixed, we suppress the subscript and write $\mathcal{D}_N=\mathcal{D}$.

\subsection*{Acknowledgments}
 Y. Deng was supported by China Postdoctoral Science Foundation (Grant No. 2025M774191) and the NSF grant of China (No. 12501117). B. Di was supported by the China Postdoctoral Science Foundation (Grant Nos. GZB20230812 \& 2024M753436 \& YJC20250892) and the National Natural Science Foundation of China (Grant Nos. 12471232 \& 12501130). D. Yan was supported in part by the National Key R\&D Program of China [Grant No. 2023YFC3007303], National Natural Science Foundation of China [Grant Nos. 12071052 \& 12271501]. K. Yang was supported by the NSF grant of China (No.12301296, 12371244), Natural Science Foundation of Chongqing, China (CSTB2024NSCQ-MSX0867, CSTB2024NSCQ-LZX0139), and the Science and Technology Research Program of Chongqing Municipal Education Commission (Grant No.KJZD-K202400505).

\section{Preliminary long-time Strichartz estimates on waveguides and examples}

In this section, we prove Theorem \ref{thm:BCP RmTn} and Theorem \ref{thm:example}.

\subsection{Preliminary long-time Strichartz estimate}
Deriving from the global-in-time Strichartz estimate in \cite{Barron2021}, we can get the long time version Strichartz estimate.

\begin{proof}[\textbf{Proof of Theorem \ref{thm:BCP RmTn}}]
As we know that the $L^2-$mass is conserved, we have 
\begin{equation*}
 \|\ed P_{\le N} \phi\|_{L_{t,x}^2([0,T] \times\bR^m\times \bT^n )}\lea T^{\frac{1}{2}}\|\phi\|_{L^2(\bR^m\times \bT^n)}.
\end{equation*}
Note that the frequency of the initial data $\phi$ is $\le N$, we  have
\begin{equation*}
 \|\ed P_{\le N} \phi\|_{L_{t,x}^{\infty}([0,T] \times\bR^m\times \bT^n )}\lea \|\widehat{P_{\le N}\phi}\|_{L^1(\bR^m\times \bZ^n)}  \lea N^{\frac{d}{2}}\|\phi\|_{L^2(\bR^m\times \bT^n)},
\end{equation*}

By the H\"{o}lder's inequality, 
\begin{align*}
 \|\ed P_{\le N} \phi\|_{L_{t,x}^p([0,T] \times\bR^m\times \bT^n )}
& = \l(\sum_{\gamma =0}^{T-1}\|\ed P_{\le N} \phi\|_{L_{t,x}^p([\gamma,\gamma+1] \times\bR^m\times \bT^n )}^p \r)^{\frac{1}{p}}  \nonumber 
\\
&
\leq T^{\frac{1}{p}-\frac{1}{q}}  \l(\sum_{\gamma =0}^{T-1}\|\ed P_{\le N} \phi\|_{L_{t,x}^p([\gamma,\gamma+1] \times\bR^m\times \bT^n )}^q \r)^{\frac{1}{q}}.
\end{align*}
Applying  \eqref{global in time strichartz estimate RmTn} to the endpoint case $p= 2+\frac{4}{d}$, $q=q(p)=\frac{4p}{m(p-2)}, s=\varepsilon>0$, we obtain
\begin{equation*}
\|\ed P_{\le N} \phi\|_{L_{t,x}^p([0,T] \times\bR^m\times \bT^n )}\lea  N^\varepsilon T^{\frac{m+2}{2p}-\frac{m}{4}} \|\phi\|_{L^2(\bR^m\times \bT^n)},
\end{equation*}
interpolating with case $p=2$, we get \eqref{eq:long time strichartz estimate RmTn} for $2< p \le 2+\frac{4}{d}.$

Applying  \eqref{global in time strichartz estimate RmTn} to $2+\frac{4}{d}<p\le 2+\frac{4}{m}$, $q=q(p)=\frac{4p}{m(p-2)}$, we obtain
\begin{align}
 \|\ed P_{\le N} \phi\|_{L_{t,x}^p([0,T] \times\bR^m\times \bT^n )}
\lea T^{\frac{m+2}{2p}-\frac{m}{4}} N^{\frac{d}{2}-\frac{d+2}{p}} \|\phi\|_{L^2(\bR^m\times \bT^n)}.\nonumber
\end{align}
Thus \eqref{eq:long time strichartz estimate RmTn} holds for $2+\frac{4}{d}<p<2+\frac{4}{m}.$
In particular, for $p=2+\frac{4}{m}$, 
\begin{equation*}
 \|\ed P_{\le N} \phi\|_{L_{t,x}^p([0,T] \times\bR^m\times \bT^n )}\lea  N^{\frac{d}{2}-\frac{d+2}{p}} \|\phi\|_{L^2(\bR^m\times \bT^n)},
 \end{equation*}
interpolating with case $p=\infty$, we get  
\eqref{eq:long time strichartz estimate RmTn} for $p\ge 2+\frac{4}{m}.$

\end{proof}

\subsection{Examples}
In this subsection, we provide some examples to suggest that, up to the $N^{\varepsilon}$-loss, the sharp constant in the long-time Strichartz inequality \eqref{eq:long time strichartz estimate} might be $C(p,T,N)$ stated in \eqref{eq:sharp constant RmTn}.
\begin{proof}[\textbf{Proof of Theorem \ref{thm:example}}]
    The proof mainly relies on the stationary phase analysis. For convenience, we introduce the notations
    \[LHS:= \|e^{it\Delta_{\bR^m\times \bT^n}} \phi\|_{L_{t,x}^p([0,T]\times \bR^m \times \bT^n)}, \quad RHS:= \|\phi\|_{L_{x}^2(\bR^m\times \bT^n)}.\]
    First, we choose the Schwartz function $\phi_1(x) \in L^2(\bR^m\times \bT^n)$ which satisfies
    \[\widehat{\phi}_1 (\xi)= \chi_{[-N,N]^{m}} (\xi_1) \chi_{[-N,N]^n}(\xi_2), \quad \xi=(\xi_1, \xi_2) \in \bR^m \times \bZ^n,\]
    then, by the stationary phase method, there exists a small constant $c>0$ such that
    \[\l|e^{it\Delta_{\bR^m \times \bT^n}} \phi_1(x)\r| \gtrsim N^{m+n}, \quad \forall (t,x) \in [0,cN^{-2}] \times \l([-cN^{-1}, cN^{-1}]^m \times [-cN^{-1}, -cN^{-1}]^n\r).\]
    Therefore, we have
    \[LHS \gtrsim N^{d-\frac{d+2}{p}}, \quad RHS = N^{\frac{d}{2}},\]
    which gives the desired constant. Second, we choose the Schwartz function $\phi_2(x) \in L^2(\bR^m\times \bT^n)$ which satisfies
    \[\widehat{\phi}_2 (\xi)= \chi_{[0,T^{-1/2}]^{m}} (\xi_1) \ca_{=0}(\xi_2), \quad \xi=(\xi_1,\xi_2) \in \bR^m \times \bZ^n,\]
    then there exists a small constant $c>0$ such that
    \[\l|e^{it\Delta_{\bR^m \times \bT^n}} \phi_2(x) \r|\gtrsim T^{-m/2}, \quad \forall (t,x) \in [0,cT] \times \l([-c\sqrt{T}, c\sqrt{T}]^m \times [0,c]^n \r).\]
    Therefore, we have
    \[LHS \gtrsim T^{\frac{2+m}{2p} -\frac{m}{2}}, \quad RHS = T^{-\frac{m}{4}},\]
    which gives the desired constant. Third, we choose the the function $\phi_3(x,y) \in L^2(x,y)$ which satisfies
    \[\widehat{\phi}_3 (\xi)= \widehat{\phi}_{3,1}(\xi_1) \widehat{\phi}_{3,2}(\xi_2) := \chi_{[0,T^{-1/2}]^m}(\xi_1) \chi_{[0,N]^n}(\xi_2), \quad (\xi_1,\xi_2) \in \bR^m \times \bZ^n,\]
    then there exists a small constant $c>0$ such that
    \[\l|e^{it\Delta_{\bR^m}} \phi_{3,1}(x_1)\r| \gtrsim T^{-m/2}, \quad \forall (t,x_1) \in [0,cT]\times [-c\sqrt{T},c\sqrt{T}]^m,\]
    and
    \[\l| e^{it\Delta_{\bT^n}} \phi_{3,2}(x_2)\r| \gtrsim N^n, \quad \forall (t,x_2) \in [0,c N^{-2}]\times [-cN^{-1},cN^{-1}]^n.\]
    Thus we conclude
    \begin{align*}
        LHS &\geq \l\| \|e^{it\Delta_{\bR^m}} \phi_{3,1}\|_{L^p([0,T^{1/2}]^m)} \|e^{it\Delta_{\bT^n}} \phi_{3,2}\|_{L^p([0,N^{-1}]^n)} \r\|_{L_t^p([0,T])} \\
        &\gtrsim T^{\frac{m}{2p} -\frac{m}{2}} \l\| e^{it\Delta_{\bT^n}}\phi_{3,2} \r\|_{L^p([0,T] \times [0,N^{-1}]^n)} \\
        & \gtrsim T^{\frac{m}{2p} -\frac{m}{2}} T^{1/p} \l\| e^{it\Delta_{\bT^n}} \phi_{3,2} \r\|_{L^p([0,1]\times [0,N^{-1}]^n)},
    \end{align*}
    where we have used the periodic property of $e^{it\Delta_{\bT^n}}$. Next, similar to the first case, we can obtain the following estimate
    \[\l\| e^{it\Delta_{\bT^n}} \phi_{3,2} \r\|_{L^p([0,1]\times [0,N^{-1}]^n)} \gtrsim N^{n-\frac{n+2}{p}}.\]
    Therefore, we have
    \[LHS \gtrsim T^{\frac{m+2}{2p} -\frac{m}{2}} N^{n-\frac{n+2}{p}}, \quad RHS= T^{-m/4} N^{n/2},\]
    which gives the desired constant. Hence we complete the proof.
\end{proof}

\section{The case $p=4$ and improved long-time Strichartz estimate}
In this section, we prove Theorem \ref{thm:long-time Strichartz for p=4} and Theorem \ref{thm:improved long-time Strichartz estimate}.

\subsection{$L^2\to L^4$ long-time Strichartz estimate on $\bR\times \bT^n$}

We mainly follow the technical framework of \cite{DPST,TT2001} to reduce the proof of the $L^2 \to L^4$-type Strichartz estimates to establishing measure estimates, which are readily manageable in our proof.

\begin{proof}[\textbf{Proof of Theorem \ref{thm:long-time Strichartz for p=4}}]
 Using the standard argument in the proof of Proposition 3.7 in \cite{DPST}, it suffices to prove the following measure estimate: for fixed $a, b \in \bR\times \bZ^n$ with $|a|, |b|\lea N$, 
 $$ \l|\l\{\xi\in\bR\times \bZ^n: |(\xi-a)\cdot(\xi-b)|\le T^{-1}, |\xi|\lea N          \r\} \r|\leaa T^{-\frac12}N^{n-2}+T^{-1} N^{n-1},  $$
 where $|E|$ denotes the standard measure of $E$ on $\R\times\bZ^n$, which is the product of  Lebesgue measure on $\R$ and the counting measure on $\bZ^n$. 
By translation, it suffices to prove the following estimate:
$$ \sup_{C\in \bR, |C|\lea N^2} \l|\l\{\xi\in\bR\times \bZ^n: ||\xi|^2-C|\le T^{-1}, |\xi|\lea N          \r\} \r|\leaa T^{-\frac12}N^{n-2}+T^{-1} N^{n-1}.    $$
We fix $C\in \bR, |C|\lea N^2$. Then
\begin{align*}
\l|\l\{\xi\in\bR\times \bZ^n: ||\xi|^2-C|\le T^{-1}, |\xi|\lea N          \r\} \r| 
\le \sum_{A\in \bZ, 0\le A\lea N^2} |X_A| |Y_A|,
\end{align*}
where
$$X_A=\l\{ \xi_1\in\bR: ||\xi_1|^2-A-C|\le  T^{-1}, |\xi_1|\lea N  \r\},$$
and
$$Y_A=\l\{ \w{\xi}=(\xi_2, \cdots, \xi_{n+1})\in\bZ^n: |\w{\xi}|^2=A, |\w{\xi}|\lea N  \r\}.$$
Note that the number of lattice points on the circle $x^2+y^2=K$ is $O(K^\varepsilon)$, we deduce that
$$ \sup_{A\in \bZ, 0\le A\lea N^2} |Y_A|\leaa N^{n-2}.  $$
So we only need to prove that
\begin{equation}\label{eq:measure estimate}
  \sum_{A\in \bZ, 0\le A\lea N^2} |X_A| \lea T^{-\frac12}+T^{-1} N.  
\end{equation}
Observe that for any $A$ satisfying $|A+C| \lea T^{-1}$ (and there are at most $O(1)$ such $A$), we have $$|X_A| \lea T^{-\frac12}.$$
For $A$ with $A+C \gg T^{-1}$, it then follows that $$|X_A|\le\sqrt{A+C+T^{-1}}-\sqrt{A+C-T^{-1}} \lea T^{-1}(A+C)^{-\frac{1}{2}}.$$
The proof of \eqref{eq:measure estimate} is completed by summing $|X_A|$ over all indices $A$ for which $|A+C|\lea N^2$.

\end{proof}

\begin{rem}
In summary, combining Theorem 1.3 with Theorem 1.1, we confirm that the optimal constant $C(p,T,N)$ stated in \eqref{eq:sharp constant RmTn} is correct when $2\le p\le 2+\frac{4}{d}$, $p=4$ or $p\ge 2+\frac{4}{m}$.
\end{rem}

\begin{rem}
Results concerning the sharp $L^2\to L^4$ Strichartz estimate for the Schr\"odinger equation on the tori or waveguide manifolds, can be found in \cite{DFZ2025,HK2024,KV2016,LZ2025,TT2001}. Studying the optimal constant in Theorem \ref{thm:long-time Strichartz for p=4} remains an interesting direction for future work.
\end{rem}

\subsection{Improved long-time Strichartz estimate}
After making some minor refinements to the level set estimation method in \cite{Bourgain1993NLS,DGG2017}, we can prove Theorem \ref{thm:improved long-time Strichartz estimate}. Before commencing the proof, we briefly explain our understanding of the level set estimation approach.

Let $F(t,x)=\ed P_{\le N} \phi(x)$, $(t,x)\in \bR\times(\bR^m \times \bT^n)$. We assume that $\|\phi\|_{L^2(\bR^m \times \bT^n)}=1$, and define
$$ E_{\lambda}= \l\{ (t,x)\in [0,T]\times (\bR^m\times \bT^n): |F(t,x)|>\lambda   \r\}, \quad \lambda>0.$$
Using Theorem \ref{thm:BCP RmTn} with $p=2+\frac{4}{d}, 2+\frac{4}{m}$ and the Chebyshev's inequality,
\begin{equation}\label{eq:estimate Elambda-1}
|E_{\lambda}| \leaa \min \l\{ T^{\frac{n}{d}} \lambda^{-(2+\frac{4}{d})}, N^{\frac{2n}{m}} \lambda^{-(2+\frac{4}{m})}   \r\}.     
\end{equation}
Next, we only need an additional bound on $|E_\lambda|$. Combining this with \eqref{eq:estimate Elambda-1} will yield an improved upper bound for $|E_\lambda|$. Furthermore, based on this upper bound for $|E_\lambda|$, we can utilize the formula
$$ \|F\|_{L_{t,x}^p([0,T]\times \bR^m \times \bT^n)}^p= p \int_0^{+\infty} \lambda^{p-1} |E_\lambda| {\rm d}\lambda, \quad p\in\l(2+\frac{4}{d}, 2+\frac{4}{m}\r),  $$
to derive the required estimate. We will employ the method from \cite{Bourgain1993NLS, DGG2017} to analyze the kernel function 
$$K_N(t,x)=\int_{\bR^m\times \bZ^n} e^{2\pi i(x\cdot \xi-t|\xi|^2)} \chi\l(\frac{\xi_1}{N}\r)\cdots \chi\l(\frac{\xi_d}{N}\r) {\rm d}\xi,$$ thereby obtaining the additional bound over $|E_\lambda|$.

Now, we invoke the dispersive bound and Weyl bound for the one-dimensional fundamental solution, and then consider the cases of $\mathbb{R}^2\times \mathbb{T}$ and $\mathbb{R}\times \mathbb{T}^2$ separately.

\begin{lem}[Dispersive bound]\label{lem:dispersive bound}
If $N\in 2^\bN$, $t\in \bR$, then
$$ \l\|  \int_\bR e^{2\pi i (y\eta-t|\eta|^2)} \chi\l(\frac{\eta}{N}\r) {\rm d}\eta    \r\|_{L^\infty_y(\bR)} \lea \min\{N, |t|^{-\frac12}\}.   $$
If $N\in 2^\bN$, $|t|\lea \frac{1}{N}$, then
$$ \l\|  \sum_{k \in \bZ} e^{2\pi i (yk-t|k|^2)} \chi\l(\frac{k}{N}\r)    \r\|_{L^\infty_y(\bT)} \lea \min\{N, |t|^{-\frac12}\}.   $$
\end{lem}

\begin{lem}[Weyl bound, Lemma 3.18 in \cite{Bourgain1993NLS}]\label{lem:Weyl bound}
If $N\in 2^\bN$, $|t|>\frac{4}{N}$, $a\in \bZ/\{0\}$, $q\in \{1, 2, \cdots N\}$, $\gcd(a,q)=1$ and $\l|t-\frac{a}{q}\r|\le \frac{1}{Nq}$, then 
$$ \l\|  \sum_{k \in \bZ} e^{2\pi i (yk-t|k|^2)} \chi\l(\frac{k}{N}\r)    \r\|_{L^\infty_y(\bT)} \lea \frac{N}{q^{\frac12}\l(  1+N|t-\frac{a}{q}|^{\frac12}\r)}.   $$ 
\end{lem}

After applying Lemma \ref{lem:Weyl bound} for $|t|\gg \frac{1}{N}$, our analysis actually relies only on the following simple upper bound: 
\begin{align*}
\l\|  \sum_{k \in \bZ} e^{2\pi i (yk-t|k|^2)} \chi\l(\frac{k}{N}\r)    \r\|_{L^\infty_y(\bT)} \lea & \frac{N}{q^{\frac12}\l(  1+N|t-\frac{a}{q}|^{\frac12}\r)}\\
\lea & N q^{-\frac12} \lea N \min\{1, |t|^{\frac12}\}.
\end{align*}
A refinement of the estimate may be possible via the major and minor arc decomposition introduced in \cite{Bourgain1993NLS,DGG2017}, which is particularly effective in higher dimensions.

\begin{proof}[\textbf{Proof of Theorem \ref{thm:improved long-time Strichartz estimate}}]

\underline{The case of $\bR^2\times \bT$. }
We consider the kernel function
$$K_N(t,x)=\int_{\bR^2\times \bZ} e^{2\pi i(x\cdot \xi-t|\xi|^2)} \chi\l(\frac{\xi_1}{N}\r) \chi\l(\frac{\xi_2}{N}\r) \chi\l(\frac{\xi_3}{N}\r) {\rm d}\xi.$$
By Lemma \ref{lem:dispersive bound} and Lemma \ref{lem:Weyl bound}, we have
\begin{equation*}
 |K_N(t,x)| \lea \begin{cases}
     \min\{N^3, |t|^{-\frac{3}{2}} \}, & \text{when~} |t|\lea \frac{1}{N}, \\
     N |t|^{-{\frac12}}, & \text{when~} \frac1N \ll|t|\lea 1, \\
   N |t|^{-1}  , & \text{when~} |t|\gg 1.
 \end{cases}   
\end{equation*}
Let $\psi$ be a good function supported on $[-2,2]$ such that $\psi >1$ on $[-1,1]$ and $\widehat{\psi}\ge 0$. Then we decompose $\psi\l(\frac{t}{T}\r) K_N(t,x) $ into
$$  \psi\l(\frac{t}{T}\r) K_N(t,x)=J_1(t,x)+ J_2(t,x)+ J_3(t,x)+J_4(t,x),  $$
where
$$  J_1(t,x)=   \psi\l(\frac{t}{T}\r) \chi\l(\frac{t}{A}\r) K_N(t,x),   $$
$$  J_2(t,x)=   \psi\l(\frac{t}{T}\r)\chi\l(Nt\r) \l[1-\chi\l(\frac{t}{A}\r) \r] K_N(t,x),   $$
$$  J_3(t,x)=   \psi\l(\frac{t}{T}\r)\chi\l(t\r) \l[1-\chi\l(Nt\r) \r] K_N(t,x),   $$
$$  J_4(t,x)=   \psi\l(\frac{t}{T}\r) \l[1-\chi\l(t\r) \r] K_N(t,x),   $$
and $A\in (0, \frac{1}{N})$ is a number to be fixed later.
By computation, we get
$$ \|\widehat{J_1}\|_{L^\infty} \lea A, $$
and
$$ \|J_2\|_{L^\infty}+ \|J_3\|_{L^\infty}+\|J_4\|_{L^\infty}\lea A^{-\frac32}+N^{\frac32}+N\lea A^{-\frac32}. $$
Following \S 4.3 Step 2 in \cite{DGG2017}, we have
\begin{align*}
\lambda^2|E_\lambda|^2\lea & \|\widehat{J_1}\|_{L^\infty} |E_\lambda|+ (\|J_2\|_{L^\infty}+ \|J_3\|_{L^\infty}+\|J_4\|_{L^\infty})|E_\lambda|  \\
\lea & A|E_\lambda|+ A^{-\frac32}|E_\lambda|^2. 
\end{align*}
We explicitly write out the implicit constant in the above expression, that is, 
\begin{equation}\label{eq:estimate Elambda-2} 
 \lambda^2|E_\lambda|^2 \le C_0   (A|E_\lambda|+ A^{-\frac32}|E_\lambda|^2),
\end{equation}
where $C_0>1$ is a constant.
For $\lambda> (2C_0)^\frac12 N^{\frac34}$, we choose $A=(2C_0)^\frac23 \lambda^{-\frac{4}{3}}$, then \eqref{eq:estimate Elambda-2} becomes 
$$ |E_\lambda|\lea \lambda^{-\frac{10}{3}}.  $$
For $p\in (\frac{10}{3}, 4)$, combining \eqref{eq:estimate Elambda-1}, we obtain
\begin{align*}
 &\|F\|_{L_{t,x}^p([0,T]\times \bR^2 \times \bT)}^p
 = p \int_0^{+\infty} \lambda^{p-1} |E_\lambda| {\rm d}\lambda \\
 \leaa &\int_0^{(2C_0)^\frac12 N^{\frac34}} \min\{T^{\frac13} \lambda^{p-\frac{10}{3}-1}, N \lambda^{p-4-1}  \} {\rm d}\lambda + \int_{(2C_0)^\frac12 N^{\frac34}}^{+\infty} \min\{ \lambda^{p-\frac{10}{3}-1}, N \lambda^{p-4-1}  \} {\rm d}\lambda \\
 \leaa &T^{2-\frac{p}2} N^{\frac32 p -5} \ca_{T> N^{\frac32}}+ T^{\frac13} N^{\frac34 p - \frac52}  \ca_{T\le N^{\frac32}} +N^{\frac32 p -5}.
\end{align*}

 \underline{ The case of $\bR\times \bT^2$.}
We consider the kernel function
$$K_N(t,x)=\int_{\bR\times \bZ^2} e^{2\pi i(x\cdot \xi-t|\xi|^2)} \chi\l(\frac{\xi_1}{N}\r) \chi\l(\frac{\xi_2}{N}\r) \chi\l(\frac{\xi_3}{N}\r) {\rm d}\xi.$$
By Lemma \ref{lem:dispersive bound} and Lemma \ref{lem:Weyl bound}, we have
\begin{equation*}
 |K_N(t,x)| \lea \begin{cases}
     \min\{N^3, |t|^{-\frac{3}{2}} \}, & \text{when~} |t|\lea \frac{1}{N}, \\
     N^2 |t|^{{\frac12}}, & \text{when~} \frac1N \ll|t|\lea 1, \\
   N^2 |t|^{-\frac12}  , & \text{when~} |t|\gg 1.
 \end{cases}   
\end{equation*}
Decompose $\psi\l(\frac{t}{T}\r) K_N(t,x) $ into
$$  \psi\l(\frac{t}{T}\r) K_N(t,x)=J_1(t,x)+ J_2(t,x)+ J_3(t,x)+J_4(t,x),  $$
where
$$  J_1(t,x)=   \psi\l(\frac{t}{T}\r) \chi\l(\frac{t}{A}\r) K_N(t,x),   $$
$$  J_2(t,x)=   \psi\l(\frac{t}{T}\r)\chi\l(Nt\r) \l[1-\chi\l(\frac{t}{A}\r) \r] K_N(t,x),   $$
$$  J_3(t,x)=   \psi\l(\frac{t}{T}\r)\chi\l(t\r) \l[1-\chi\l(Nt\r) \r] K_N(t,x),   $$
$$  J_4(t,x)=   \psi\l(\frac{t}{T}\r) \l[1-\chi\l(t\r) \r] K_N(t,x),   $$
and $A\in (0, \frac{1}{N})$ is a number to be fixed later. By computation, we get
$$ \|\widehat{J_1}\|_{L^\infty} \lea A, $$
and
$$ \|J_2\|_{L^\infty}+ \|J_3\|_{L^\infty}+\|J_4\|_{L^\infty}\lea A^{-\frac32}+N^{2}. $$
Thus, we have
\begin{equation}\label{eq:estimate Elambda-3}
    \lambda^2|E_\lambda|^2\le C_0 \l(A|E_\lambda|+ (A^{-\frac32}+N^2)|E_\lambda|^2 \r), 
\end{equation}
where $C_0>1$ is a constant. For $\lambda> (4C_0)^\frac12 N$, we choose $A=(4C_0)^\frac23 \lambda^{-\frac{4}{3}}$, then \eqref{eq:estimate Elambda-3} becomes 
$$ |E_\lambda|\lea \lambda^{-\frac{10}{3}}.  $$
By applying Theorem \ref{thm:BCP RmTn} with $p=\frac{10}{3}, 6$ and Theorem \ref{thm:long-time Strichartz for p=4}, combined with Chebyshev’s inequality, we obtain
$$ |E_{\lambda}| \leaa \min \l\{ T^{\frac{2}{3}} \lambda^{-\frac{10}{3}}, (T^{\frac12}+N) \lambda^{-4} , N^{4} \lambda^{-6}   \r\}.  $$
For $p\in (4,6)$, we have
\begin{align*}
 &\|F\|_{L_{t,x}^p([0,T]\times \bR^2 \times \bT)}^p
 = p \int_0^{+\infty} \lambda^{p-1} |E_\lambda| {\rm d}\lambda \\
 \leaa &\int_0^{(4C_0)^\frac12 N} \min\{(T^{\frac12}+N) \lambda^{p-4-1} , N^{4} \lambda^{p-6-1}  \} {\rm d}\lambda + \int_{(4C_0)^\frac12 N}^{+\infty} \min\{ \lambda^{p-\frac{10}{3}-1}, N^{4} \lambda^{p-6-1}    \} {\rm d}\lambda \\
 \leaa & T^{\frac12} N^{p-4} \ca_{N^2< T\le N^4}+ T^{\frac32-\frac{p}4} N^{2p-8} \ca_{T> N^4}  +N^{\frac32 p -5}.
\end{align*}
For $p\in (\frac{10}{3},4)$, we have
\begin{align*}
 &\|F\|_{L_{t,x}^p([0,T]\times \bR^2 \times \bT)}^p
 = p \int_0^{+\infty} \lambda^{p-1} |E_\lambda| {\rm d}\lambda \\
 \leaa &\int_0^{(4C_0)^\frac12 N} \min\{T^{\frac{2}{3}} \lambda^{p-\frac{10}{3}-1},(T^{\frac12}+N) \lambda^{p-4-1}   \} {\rm d}\lambda + \int_{(4C_0)^\frac12 N}^{+\infty} \min\{ \lambda^{p-\frac{10}{3}-1}, N^{4} \lambda^{p-6-1}    \} {\rm d}\lambda \\
 \leaa & T^{\frac23} N^{p-\frac{10}{3}}  \ca_{1\le T\le N^{\frac12}}+T^{4-p} N^{\frac32 p -5} \ca_{N^{\frac12}<T\le N^2}  +T^{\frac32-\frac{p}4} \ca_{T>N^2}  +N^{\frac32 p -5}.
\end{align*}

\end{proof}

\begin{rem}
Pointwise comparison of $|K_N(t,x)|$ readily shows our estimates on $\bR^2\times \bT$ and $\bR\times \bT^2$ are better than the generic $\bT^3$ estimates in \cite{DGG2017}. For higher dimensions, however, it is difficult to directly compare $\bR^m\times \bT^n$ with a generic $\bT^d$ in the same way, a phenomenon that deserves deeper study.
\end{rem}

\section{The growth of Sobolev norms of solutions to NLS on 3D waveguide}
In this section, we establish Theorem~\ref{thm:growth of solutions to NLS on 3D waveguide}. We first recall that the NLS equation \eqref{nls} is globally well-posed in $H^1\left(\bR\times \bT^2\right)$, a result that follows from a standard application of local well-posedness theory combined with the conservation of energy. Furthermore, by the persistence of regularity, solutions emanating from initial data in $H^s(\bR\times \bT^2)$ with $s>1$ remain in $H^s$ for all time.

Adopting the framework established in \cite{deng2019growth}, we introduce the following family of space-time norms essential for the proof of Theorem~\ref{thm:growth of solutions to NLS on 3D waveguide}. First, we define the $X^s$ and $Y^s$ spaces adapted to the linear flow. Let $\widehat{u}(t, \xi)$ denote the spatial Fourier transform of $u$ on $\mathbb{R}\times \mathbb{Z}^2$. We define
$$\|u\|_{X^s} := \left(\int_{\mathbb{R}\times \mathbb{Z}^2}\langle \xi\rangle^{2 s}\left\|e^{i |\xi|^2 t} \widehat{u}(t, \xi)\right\|_{U^2_t}^2 \mathrm{d}\xi \right)^{\frac{1}{2}}$$
and
$$\|u\|_{Y^s} := \left(\int_{\mathbb{R}\times \mathbb{Z}^2}\langle \xi\rangle^{2 s}\left\|e^{i |\xi|^2 t} \widehat{u}(t, \xi)\right\|_{V^2_t}^2 \mathrm{d}\xi \right)^{\frac{1}{2}}.$$
For the precise definitions and properties of the atomic spaces $U^2_t$ and $V^2_t$, we refer the reader to \cite{deng2019growth,HTT1}.

We also invoke the standard Bourgain norms $X^{s, b}$. Let $\mathcal{F}_{t, x} u$ denote the space-time Fourier transform. We define:
$$\|u\|_{X^{s, b}} := \left(\int_{\mathbb{R}\times \mathbb{Z}^2} \int_{\mathbb{R}}\langle\xi\rangle^{2 s}\langle\tau+|\xi|^2\rangle^{2 b}\left|\mathcal{F}_{t, x} u(\tau, \xi)\right|^2 \mathrm{d} \tau \mathrm{d} \xi\right)^{1 / 2}.$$
For any finite time interval $I$, the restriction norm is given by$$\|u\|_{X^{s, b}(I)} := \inf \left\{ \|v\|_{X^{s, b}} : v \equiv u \text{ on } I \right\}.$$
Finally, we introduce the long-time Strichartz norms, as defined in \cite{deng2019growth}, to capture averaged spacetime integrability over unit time intervals.
\begin{definition}[Long-time Strichartz norms]\label{longtimenorm}Let $J \subset \mathbb{R}$ be a finite time interval and fix parameters $7/2\leq q\leq 4$ and $5\leq \tilde{q}\leq 12$. For a frequency parameter $N \geq 1$, we define
\begin{equation}\label{longnorm}
\|u\|_{S_{N,J}^{q,\tilde{q}}}
:=
\bigg(
\sum_{m\in\mathbb{Z}}
\Big(
N^{\frac{5}{\tilde{q}}-\frac{1}{2}}
\|u\|_{L_{t,x}^{\tilde{q}}([m,m+1]\cap J)}
\Big)^{q}
\bigg)^{\frac{1}{q}}.
\end{equation}
Furthermore, we define the maximal norm
\begin{equation}\label{longnorm2}
\|u\|_{S_{N,J}^{q}}
:=
\bigg(
\sum_{m\in\mathbb{Z}}
\Big(
\sup_{5\leq \tilde{q}\leq 12}
N^{\frac{5}{\tilde{q}}-\frac{1}{2}}
\|u\|_{L_{t,x}^{\tilde{q}}([m,m+1]\cap J)}
\Big)^{q}
\bigg)^{\frac{1}{q}}.
\end{equation}
By H\"older's inequality, we observe the equivalenc $\|u\|_{S_{N,J}^{q}}
\sim
\max\left(
\|u\|_{S_{N,J}^{q,5}},
\,
\|u\|_{S_{N,J}^{q,12}}
\right).$
We note that both $\|u\|_{S_{N,J}^{q}}$ and $\|u\|_{S_{N,J}^{q,\tilde{q}}}$ are monotone decreasing with respect to $q$. Henceforth, whenever these norms are utilized, we assume $u$ is spectrally localized at frequency $|\xi| \sim N$.
\end{definition}

We now collect several analytical tools that will be used repeatedly in the sequel. 
\subsection{Tools}
The estimates presented below play a fundamental technical role in our arguments. 
Their proofs follow verbatim from the corresponding results in \cite{deng2019growth,deng2019growth2}, 
with straightforward adaptations to the setting of $\mathbb{R}\times\mathbb{T}^2$.
\begin{prop} \label{linear} Let $I$ be a time interval with length $|I|\lesssim 1$. We have the following estimates.
\begin{enumerate}
\item For any fixed $q>10/3$, an improved Strichartz estimate \begin{equation}\label{str2}\big\|e^{it\Delta}P_\mathcal{R}f\big\|_{L_{t,x}^q(I)}\lesssim N^{\frac{3}{2}-\frac{5}{q}}\bigg(\frac{|\mathcal{R}|}{N^3}\bigg)^{\frac{1}{2}-\frac{5}{3q}-\varepsilon}\|P_{\mathcal{R}}f\|_{L^2}\end{equation} holds for any fixed $\varepsilon>0$ and any strip  $\mathcal{R}:=\left\{\xi \in \mathbb{R}\times\mathbb{Z}^2:|\xi|\lea N, |a \cdot \xi-A| \leq M\right\}$
with some $ a \in \mathbb{R}^3,|a|=1, A \in \mathbb{R}$ and $M\leq N$.
\item Embeddings: we have \begin{equation}\label{embed3}\|u(t)\|_{H_x^s}\lesssim \|u\|_{X^s(I)}\end{equation} for any $t\in I$ if $u$ is weakly left continuous in $t$.
\item For any $s$ and $b\in (1/2,1)$, any fixed smooth cutoff function $\chi$, 
let \[\mathcal{I}'G(t)=\chi(t-m)\int_m^te^{-i(t-t')}G(t')\,\mathrm{d}t'\] be the smoothly truncated Duhamel operator, then \begin{equation}\label{xsb02}\|\mathcal{I}'G\|_{X^{s,b}}\lesssim\|G\|_{X^{s,b-1}}\end{equation} 

\item Transfer principle: \begin{equation}\label{strxsb}\|P_Nu\|_{L_{t,x}^q(I)}\lesssim N^{\frac{3}{2}-\frac{5}{q}}\|u\|_{X^{0,b}}\end{equation} holds for any fixed $b>1/2$ and $q>10/3$. The same estimate is true for $P_{\mathcal{B}}u$ if $\mathcal{B}$ is any ball of radius $N$, as well as the improvement for $P_{\mathcal{R}}u$ as in (\ref{str2}) when $\mathcal{R}$ has diameter not exceeding $N$.
\item For any fixed $\varepsilon>0$ and any fixed smooth cutoff $\chi$, we have
\begin{equation}\label{xsb1}\|\chi(t)u\|_{X^{s,1/4-\varepsilon}}\lesssim \|u\|_{X^s}.\end{equation}
\end{enumerate}
\end{prop}

\begin{prop}[A long-time Strichartz estimate]\label{longstr} One has
\begin{equation}\label{strlong}\sup_{|J|=N^\nu}\|e^{it\Delta}P_Nf\|_{S_{N,J}^{q,\tilde{q}}}\lesssim N\|P_Nf\|_{L^2}.\end{equation} uniformly for $7/2\leq q\leq 4$, $5\leq \tilde{q}\leq 12$, $0\leq \nu\leq 1/10$, and any $N$. In particular one has \begin{equation}\label{strlong2}\sup_{|J|=N^\nu}\|e^{it\Delta}P_Nf\|_{S_{N,J}^{q}}\lesssim N\|P_Nf\|_{L^2}\end{equation} uniformly in $q$, $\nu$ and $N$.
\end{prop}
\begin{rem}
In fact, following the approach in \cite[Proposition 2.7]{deng2019growth}, we utilize the conclusion of Corollary \ref{cor-long-time}. Since our long-time Strichartz estimates on the 3d waveguide are better than those in \cite{DGG2017} on the 3d irrational torus, we could improve the index of $\nu$ in \eqref{strlong} and \eqref{strlong2}. However, the current range of $\nu$ is enough for our later needs.
\end{rem}
\begin{prop}[Almost conservation of energy for $\mu=5$]\label{alm-con-ene} For $\mu=5$. Suppose $\gamma<\frac{1}{6}$. Suppose $u$ is a solution to \eqref{nls} on $\mathbb{R} \times \mathbb{R}\times \mathbb{T}^2$ such that
$$
\sup _{0 \leq t \leq T}\|\mathcal{D} u(t)\|_{H^1} \lesssim \eta, \quad T \leq N^\gamma .
$$
Then one has
$$
\begin{aligned}
& |E[\mathcal{D} u(t)]-E[\mathcal{D} u(0)]| \\
 \lesssim & \eta^2 \sum_K \min \left(1, \frac{K}{N}\right)^{1 / 6} \sup _{|J| \leq K^v, J \subset[0, T]}\left\|P_K \mathcal{D} u\right\|_{S_{K, J}^4}^4
\end{aligned}
$$
uniformly for all $0 \leq t \leq T$.
\end{prop}

\begin{prop}[Almost conservation of energy for $3<\mu<5$]\label{almostconse2} For $3<\mu<5$.
Suppose $\|u(0)\|_{H^1} \leq E$ and $\|\mathcal{D} u(0)\|_{H^1} \leq C_1 E$, for a constant $C_1>0$. Choose $\epsilon$ such that $u$ is a local solution to \eqref{nls} on $[0,\epsilon]\times \bR\times \bT^2$. Then for $0 \leq t \leq \epsilon$, we have
\begin{align*}
        &|E[\mathcal{D} u(t)]-E[\mathcal{D} u(0)]| \\
\lesssim &N^{\max (\mu-5,-1)+}+N^{o(1)} \sum_M M^{o(1)} \min \left(1, N^{-1} M\right)\left\|P_M \mathcal{D} u\right\|_{L_{t, x}^{[10 /(6-\mu)]+}\left([0, \epsilon] \times \bR\times \mathbb{T}^2\right)}.
\end{align*}

\end{prop}

\subsection{Long-Time Strichartz Bounds for the solution of \eqref{nls}}
In the preceding sections, we have established a long-time Strichartz estimate for the linear Schr\"odinger equation on \(\mathbb{R}\times \mathbb{T}^2\). Building upon this result, we can further derive an upper bound of the long-time Strichartz estimate for solutions to the nonlinear Schr\"odinger equation \eqref{nls}.
\begin{prop}\label{nonlctrl} For $\mu=5$, let $\gamma=1/300$. Let $u$ be a solution to \eqref{nls} on $\mathbb{R}\times \bR\times \bT^2$ such that \begin{equation}\label{nrgboot2}\sup_{t\in[0,T]}\| \mathcal{D} u(t) \|_{H^1} \lesssim \eta,\end{equation} with $T\leq N^\gamma$, and define \begin{equation}\label{nonlinear}A_K=A_K(T):=\sup_{|J|\leq K^{\gamma},J\subset[0,T]}\|P_K\mathcal{D}u\|_{S_{K,J}^{7/2}},\end{equation} then, if $A_K\lesssim 1$ for any dyadic $K$, then we have \begin{equation}\label{nonlinear2}A_K\lesssim\eta,\quad \sum_{K\geq N}A_K^2\lesssim\eta^2.\end{equation}
\end{prop}

\begin{prop}\label{conversation2}
For $3<\mu<5$, let \(q_0 = 10/(6-\mu)+\), and \(\sigma = 1/2-5/q_0=(\mu-5)/2+\).
Suppose that $u$ is a solution to \eqref{nls} with
$$
\sup _{0 \leq t \leq \epsilon K^\gamma}\|\mathcal{D} u(t)\|_{H^1} \lesssim 1,
$$
then we have for any $\gamma>0$ that
$$
\sum_{m=0}^{K^\gamma}\left\|P_K \mathcal{D} u\right\|_{L_{t, x}^{q_0}\left([m \epsilon,(m+1) \epsilon] \times \bR\times \bT^2\right)} \lesssim K^{\gamma+\sigma+} \cdot \max \left(K^{-\frac{\gamma}{q_0}}, K^{\gamma+\theta_1}\right),
$$
where $\epsilon$ is as stated in Proposition 4.6, and
$$
\theta_1=\frac{(\mu-3)(\mu-5)}{2(65 \mu-162)}<0 .
$$
\end{prop}

\begin{proof}[Proof of Proposition \ref{nonlctrl}] Our proof follows the line of reasoning of Proposition 3.3 in \cite{deng2019growth} verbatim. The key difference lies in the decomposition of the nonlinear term: in \cite{deng2019growth}, the author extracted the zero mode of the outcome arising from the action of four factors within the quintic nonlinear term, and applied it to the image rotation transformation of $u$—the solution to equation \eqref{nls}.  A corresponding equation for the transformed unknown function was then derived, followed by term-by-term estimates. By contrast, given that we are working in the waveguide \(\mathbb{R}\times \mathbb{T}^2\), which features one non-compact Euclidean direction, we gain a favorable frequency gain at sufficiently low frequencies. As a consequence, it is unnecessary to isolate the zero mode for the purpose of image rotation.

Notice that (\ref{nrgboot2}) implies\begin{equation}\label{localxs}\|\mathcal{D}u\|_{X^1(I)}\lesssim\eta\end{equation} for any interval $I\subset[0,T]$ with $|I|\lesssim1$. To bound $A_K$, by definition, we choose an interval $J\subset[0,T]$ with $|J|\leq K^{\gamma}$. Since $T\leq N^\gamma$, if $K\geq N$ we can assume $J=[0,T]$; if $K< N$, by translation, we will also assume that the left endpoint of $J$ is $0$.

Now that $J=[0,T']$ with $T'\leq T\leq N^\gamma$, we start with the evolution equation satisfied by $P_K\mathcal{D}u$, namely
\[(i\partial_t+\Delta)P_K\mathcal{D}u=P_K\mathcal{D}(|u|^4u).\] 

 Define 
 \begin{equation}\label{defv}
 v(t)=P_K\mathcal{D}u(t),
 \end{equation} 
 then by Littlewood-Paley decomposition, we have\[(i\partial_t+\Delta)v=\mathcal{N}_1+\mathcal{N}_2+\mathcal{N}_3+\mathcal{N}_4,\] 
 where
 \begin{align*}
    \mathcal{N}_1&=\sum_{\text{at least two $K_j\gtrsim K$}} P_K\mathcal{D}(P_{K_1}u\cdot \overline{P_{K_2}u}\cdot P_{K_3}u\cdot\overline{P_{K_4}u}\cdot P_{K_5}u)\\
    \mathcal{N}_2&=3\sum_{\substack{K_1\sim K\\
K^{\alpha_1}\lesssim\max_{l\neq 1}K_l\ll K}}P_K\mathcal{D}(P_{K_1}u\cdot \overline{P_{K_2}u}\cdot P_{K_3}u\cdot\overline{P_{K_4}u}\cdot P_{K_5}u)\\
\mathcal{N}_3&=2\sum_{\substack{K_2\sim K\\
\max_{l\neq 2}K_l\ll K^{\alpha_1}}}P_K\mathcal{D}(P_{K_1}u\cdot \overline{P_{K_2}u}\cdot P_{K_3}u\cdot\overline{P_{K_4}u}\cdot P_{K_5}u)\\
\mathcal{N}_4&=3\sum_{\substack{K_1\sim K\\
\max_{l\neq 1}K_l\ll K^{\alpha_1}}}P_K\mathcal{D}(P_{K_1}u\cdot (\overline{P_{K_2}u}\cdot P_{K_3}u\cdot\overline{P_{K_4}u}\cdot P_{K_5}u)).
\end{align*}
 By Duhamel's formula
\begin{equation}\label{newduham}v(t)=e^{it\Delta}P_K\mathcal{D}u(0)-i\int_0^t e^{i(t-t')\Delta}\big(\mathcal{N}_1(t')+\mathcal{N}_2(t')+\mathcal{N}_3(t')+\mathcal{N}_4(t')\big)\,\mathrm{d}t'.
\end{equation} 
For $0\leq t\leq T'$, we write 
\[v(t)=v_{\mathrm{lin}}(t)-i\sum_{0\leq m\leq T'}(v_m'+v_m'')(t),\] 
where $v_{\mathrm{lin}}(t)=e^{it\Delta}P_K\mathcal{D}u(0)$, and 
\[v_m'(t)=\mathbf{1}_{[m,m+1)}(t)\int_m^t e^{i(t-t')\Delta}\big(\mathcal{N}_1(t')+\mathcal{N}_2(t')+\mathcal{N}_3(t')+\mathcal{N}_4(t')\big)\,\mathrm{d}t',\]
 and 
 \(v_m''(t)=\mathbf{1}_{[m+1,+\infty)\cap J}(t)\cdot e^{it\Delta}g_m\) 
 with 
 \[g_m=\int_m^{m+1}e^{-it'\Delta}\big(\mathcal{N}_1(t')+\mathcal{N}_2(t')+\mathcal{N}_3(t')+\mathcal{N}_4(t')\big)\,\mathrm{d}t'.\] 
 Note that, when $[m,m+1]\not\subset[0,T']$ (this can happen for at most two $m$), we should replace the interval $[m,m+1]$ by $[m,m+1]\cap [0,T']$. By Proposition \ref{longstr}, we know that 
 \[\|v_{\mathrm{lin}}\|_{S_{K,J}^{7/2}}\lesssim K\|P_K\mathcal{D}u(0)\|_{L^2}\lesssim a_K,\] 
 where\[a_K=\left\{\begin{split}&K\|P_K\mathcal{D}u(0)\|_{L^2},&K\geq N,\\
&\eta,&K< N.\end{split}\right.\]
 Next, we estimate $v_m'$ and $v_m''$ for each fixed $m$. We analyze the contributions from the nonlinear terms $\mathcal{N}_j$ separately. The estimates for $\mathcal{N}_1$ and $\mathcal{N}_2$ are identical to those of the corresponding terms in the proof of Proposition 3.3 in \cite{deng2019growth}, and thus their estimates are omitted here. We proceed to establish the estimates for terms $\mathcal{N}_3$ and $\mathcal{N}_4$. These terms will be estimated using $X^{s,b}$ spaces, combined with resonance analysis and the improved estimate (\ref{str2}) for strips. Note that, in terms of resonance $\mathcal{N}_3$ is better than $\mathcal{N}_4$, thus we will focus on the estimate of $\mathcal{N}_4$.

 We denote by the associated quantities $v_{m,K_1,\cdots,K_5}'$, $v_{m,K_1,\cdots,K_5}''$ and $g_{m,K_1,\cdots,K_5}$  the portion of $\mathcal{N}_4$ corresponding to the functions $v_m', v_m'', g_m$ restricted to the dyadic multi-frequency $(K_1,\cdots,K_5)$.
Since these functions depend on $u$ only within the time interval $[m,m+1]$ and satisfy the local bound $\|\mathcal{D}u\|_{X^1[m-1,m+2]}\lesssim\eta$, we may assume, without loss of generality by applying a suitable extension, that the global bound $\|\mathcal{D}u\|_{X^1}\lesssim \eta$ holds.

 We choose $\alpha_1=\frac{1}{200}$ below, and by (\ref{xsb1}), we have that $\|\mathcal{D}u\|_{X^{1,1/4-\varepsilon}}\lesssim\eta$ for any fixed $\varepsilon>0$. By symmetry, we may assume $K_2\geq\cdots \geq K_5$.
We shall prove that 
\[\|v_{m,K_1,\cdots,K_5}'\|_{S_{K,J}^{7/2}}+\|v_{m,K_1,\cdots,K_5}''\|_{S_{K,J}^{7/2}}\lesssim \eta^5K^{-1/1000}\]
 for any $(K_1,\cdots,K_5)$ satisfying $K_2\geq\cdots \geq K_5$, so that the logarithmic factor coming from summation in $(K_1,\cdots,K_5)$ can be omitted. Consider the function 
 \[\rho(t,x)=\prod_{j=2}^5P_{K_j}\widetilde{u}.\]
 By inserting a suitable time cutoff $\chi(t-m)$ we can write \[\chi(t-m)\rho(t,x)=\chi(t-m)\int_{|\xi|\leq K^{\alpha_1}}\int_{\mathbb{R}}d(\tau,\xi)e^{i( x\cdot\xi+t\tau)}\,\mathrm{d}\tau \mathrm{d}\xi .\]
 Note that $K_2\ll K^{\alpha_1}$, we use the equation \eqref{nls} to compute, for any interval $I$ containing $m$ with $|I|\lesssim 1$, that \[\begin{split}\|\partial_t\rho\|_{L_{t,x}^2(I)}&\lesssim\|P_{\leq K^{\alpha_1}}u\|_{L_{t,x}^{\infty}}^3\cdot\bigg(K^{2\alpha_1}\|P_{\leq K^{\alpha_1}}u\|_{L_{t,x}^2(I)}+\|P_{\leq K^{\alpha_1}}(|u|^4u)\|_{L_{t,x}^2(I)}\bigg)\\
&\lesssim \eta^3K^{3\alpha_1/2}(K^{2\alpha_1}\eta+K^{3\alpha_1/2}\|u\|_{L_t^\infty L_x^5(I)}^5)\lesssim \eta^4K^{4\alpha_1}.\end{split}\]
 This implies $\|\rho\|_{H_t^1H_x^2}\lesssim\eta^4K^{7\alpha_1}$, therefore 
  \begin{equation}\label{verylow}
  \int_{|\xi|\leq K^{\alpha_1}}\int_{\mathbb{R}}\langle \tau\rangle^{1/3}|d(\tau,\xi)|\,\mathrm{d}\tau \mathrm{d}\xi   \lesssim\eta^4K^{7\alpha_1}
  \end{equation}
  and
\begin{equation}
    \int_{|\xi|\leq K^{-20\alpha_1}}\int_{\mathbb{R}}\langle \tau\rangle^{1/3}|d(\tau,\xi)|\,\mathrm{d}\tau \mathrm{d}\xi   \lesssim\eta^4K^{-3\alpha_1}.
\end{equation}
  The above low-frequency $|\xi|\leq K^{-20\alpha_1}$ estimate constitutes the key distinction from the pure torus case studied in [1]. This is because we work in the waveguide \(\mathbb{R}\times \mathbb{T}^2\), where the corresponding frequency space \(\mathbb{R}\times \mathbb{Z}^2\) includes a real line \(\mathbb{R}\) direction. In the low-frequency regime considered above, this direction yields a sufficiently small measure, a favorable gain that enables us to establish the good estimate as shown. By contrast, such a gain fails to materialize in the pure torus setting, since the measure in the aforementioned low-frequency range is equal to 1, which is not small enough.
  
  Therefore, with fixed $K_1,\cdots,K_5$, we can reduce to estimating the functions
  \[h_{\tau,\xi}(t,x)=\langle \tau\rangle^{-1/3}\chi(t-m)\int_m^t e^{i(t-t')\Delta}\mathcal{D}\big(e^{i( x\cdot\xi+ t'\tau)}P_{K_1}u(t',x)\big)\,\mathrm{d}t',\] 
  where $\chi$ is a suitable cutoff function as above. We have
  \begin{align*}
  |v_{m,K_1,\cdots,K_5}'| \lea &  \int_{|\xi|\leq K^{-20\alpha_1}}   \int_{\mathbb{R}}\langle \tau \rangle^{1/3} |d(\tau,\xi)| |h_{\tau,\xi}(t,x)|\,\mathrm{d}\tau \mathrm{d}\xi \\ &+ \int_{K^{-20\alpha_1}<|\xi|\leq K^{\alpha_1}}\int_{\mathbb{R}}\langle \tau \rangle^{1/3} |d(\tau,\xi)| |h_{\tau,\xi}(t,x)|\,\mathrm{d}\tau \mathrm{d}\xi     
  \end{align*}
  and
  \begin{align*}
   |v_{m,K_1,\cdots,K_5}''|\lea &  \int_{|\xi|\leq K^{-20\alpha_1}}\int_{\mathbb{R}}\langle \tau \rangle^{1/3} |d(\tau,\xi)| |e^{i(t-m-1)\Delta}h_{\tau,\xi}(m+1,x)|\,\mathrm{d}\tau \mathrm{d}\xi \\
  &+ \int_{K^{-20\alpha_1}<|\xi|\leq K^{\alpha_1}}\int_{\mathbb{R}}\langle \tau \rangle^{1/3} |d(\tau,\xi)| |e^{i(t-m-1)\Delta}h_{\tau,\xi}(m+1,x)|\,\mathrm{d}\tau \mathrm{d}\xi.
  \end{align*}
  By Minkowski's inequality,
  \begin{align*}
  &\|v_{m,K_1,\cdots,K_5}'\|_{S_{K,J}^{7/2}}+\|v_{m,K_1,\cdots,K_5}''\|_{S_{K,J}^{7/2}}\\
  \lesssim &  \eta^4K^{-3\alpha_1}\sup_{\tau\in\mathbb{R}, |\xi|\leq K^{-20\alpha_1}}\big(\sup_{5\leq q\leq 12}K^{\frac{5}{q}-\frac{1}{2}}\|h_{\tau,\xi}\|_{L_{t,x}^{q}[m,m+1]}+\|e^{i(t-m-1)\Delta}h_{\tau,\xi}(m+1)\|_{S_{K,J}^{7/2}}\big)\\
   &+\eta^4K^{7\alpha_1}\sup_{\tau\in\mathbb{R}, K^{-20\alpha_1}<|\xi|\lesssim K^{\alpha_1}}\big(\sup_{5\leq q\leq 12}K^{\frac{5}{q}-\frac{1}{2}}\|h_{\tau,\xi}\|_{L_{t,x}^{q}[m,m+1]}+\|e^{i(t-m-1)\Delta}h_{\tau,\xi}(m+1)\|_{S_{K,J}^{7/2}}\big).
  \end{align*}
  Notice that by (\ref{embed3}), (\ref{strxsb}) and (\ref{strlong2}),
  \[\sup_{5\leq q\leq 12}K^{\frac{5}{q}-\frac{1}{2}}\|h_{\tau,\xi}\|_{L_{t,x}^{q}[m,m+1]}+\|e^{i(t-m-1)\Delta}h_{\tau,\xi}(m+1)\|_{S_{J,K}^{7/2}}\lesssim \|h_{\tau,\xi}\|_{X^{1,3/4}},\]
   and that by (\ref{xsb02}), we have 
   \[\langle \tau\rangle^{1/3}\|h_{\tau,\xi}\|_{X^{1,3/4}}\lesssim \|\nabla\mathcal{D}P_{K_1}u\|_{L_{t,x}^2}\lesssim\|\mathcal{D}P_{K_1}u\|_{X^{1,1/5}}\lesssim\eta,\]
    we see that 
    $$\eta^4K^{-3\alpha_1}\sup_{5\leq q\leq 12}K^{\frac{5}{q}-\frac{1}{2}}\|h_{\tau,\xi}\|_{L_{t,x}^{q}[m,m+1]}+\|e^{i(t-m-1)\Delta}h_{\tau,\xi}(m+1)\|_{S_{J,K}^{7/2}}\lesssim \eta^5 K^{-3\alpha_1} $$
    if $\tau\in \bR$, $|\xi|\leq K^{-20\alpha_1}$, and
    \[\eta^4K^{7\alpha_1}\sup_{5\leq q\leq 12}K^{\frac{5}{q}-\frac{1}{2}}\|h_{\tau,\xi}\|_{L_{t,x}^{q}[m,m+1]}+\|e^{i(t-m-1)\Delta}h_{\tau,\xi}(m+1)\|_{S_{J,K}^{7/2}}\lesssim \eta^5 K^{-3\alpha_1}\] 
    if $|\tau|\gtrsim K^{30\alpha_1}$, $K^{-20\alpha_1}<|\xi|\lesssim K^{\alpha_1}$. 
    
    Now we fix $\tau\in\bR, \xi\in \bR\times \bT^2$ such that $|\tau|\lea K^{30\alpha_1}$ and $K^{-20\alpha_1}<|\xi|\lesssim K^{\alpha_1}$. we shall decompose $P_{K_1}u=P_{\mathcal{R}}u+P_{K_1}P_{\mathcal{R}'}u$, where 
    \[\mathcal{R}=\big\{\w{\xi}\in \bR\times\mathbb{Z}^2:|\w{\xi}|\sim K_1,\l| |\w{\xi}+\xi|^2-|\w{\xi}|^2 \r|\lesssim K^{45\alpha_1}\big\}\] 
    and $\mathcal{R}'=\bR\times\mathbb{Z}^2-\mathcal{R}$. Clearly we have $|\mathcal{R}|\lesssim K^{2+65\alpha_1}$.

    For the term $P_\mathcal{R}u$, which differs from the corresponding term appearing in \cite{deng2019growth}, denote its contribution to $h_{\tau,\xi}$ by $h_{\tau,\xi}'$, then we have $\|h_{\tau,\xi}'\|_{X^{1,3/4}}\lesssim\eta$ as above, and moreover the spatial Fourier transform $\widehat{h_{\tau,\xi}}$ is supported in a translate of $\mathcal{R}$, so by (\ref{str2}) and the corresponding $X^{s,b}$ estimate, we have (note that, for the $S_{J,K}^{7/2}$ norm, we will reduce it to $\lesssim K^{\gamma}$ intervals of length $1$, losing a factor $K^{\gamma}$ in the process)
    \begin{align*}
    &\eta^4K^{7\alpha_1}\sup_{5\leq q\leq 12}K^{\frac{5}{q}-\frac{1}{2}}\|h_{\tau,\xi}'\|_{L_{t,x}^{q}[m,m+1]}+\|e^{i(t-m-1)\Delta}h_{\tau,\xi}'(m+1)\|_{S_{J,K}^{7/2}}\\
    \lesssim & \eta^4K^{7\alpha_1+\gamma}\bigg(\frac{K^{2+65\alpha_1}}{K^3}\bigg)^{\frac{1}{6}}\|h_{\tau,\xi}'\|_{X^{1,3/4}}\lesssim \eta^5 K^{-1/6+19\alpha_1},\end{align*} 
    using the fact that $\gamma<\alpha_1$.

   While for the term $P_{K_1}P_{\mathcal{R}'}u$, denote its contribution to $h_{k, \xi}$ by $h_{k, \xi}^*$, we can follow the exact approach taken by the author of \cite{deng2019growth} to obtain the same estimate as follows,
\begin{align*}
\eta^4 K^{7 \alpha_1}  \left(  \sup _{5 \leq q \leq 12} K^{\frac{5}{q}-\frac{1}{2}}\left\|h_{k, \xi}^*\right\|_{L_{t, x}^q[m, m+1]} 
 +\left\|e^{i(t-m-1) \Delta} h_{k, \xi}^*(m+1)\right\|_{S_{K, J}^{7 / 2}}\right) \lesssim \eta^5 K^{-\alpha_1 / 2} .
\end{align*}
 Summing up, we get that
$$
\begin{aligned}
\left\|v_{m, K_1, \ldots, K_5}^{\prime}\right\|_{S_{K, J}^{7 / 2}}+\left\|v_{m, K_1, \ldots, K_5}^{\prime \prime}\right\|_{S_{K, J}^{7 / 2}} & \lesssim \eta^5 \max \left(K^{-\alpha_1 / 2}, K^{-1 / 6+19 \alpha_1}\right) \\
& \lesssim \eta^5 K^{-1 / 400}
\end{aligned}
$$
as $\alpha_1=\frac{1}{200}$. Proceeding with the same argument as in \cite{deng2019growth}, we eventually get what we desired.
\end{proof}

\begin{proof}[Proof of Proposition \ref{conversation2}]
Fix $b=1 / 2+$. Recall $q_0=10 /(6-\mu)+$ and $\sigma=1 / 2-5 / q_0=(\mu-5) / 2+$ (notice that $-1<\sigma<0$). The assumption gives that $\|u\|_{X^{1, b}([0, \epsilon])}+\|\mathcal{D} u\|_{X^{1, b}([0, \varepsilon])} \lesssim 1$. By considering a suitable extension we may assume $\|u\|_{X^{1, b}}+\|\mathcal{D} u\|_{X^{1, b}} \lesssim 1$, and that $u$ is compactly supported in time. In this proof, we fix one parameter $\gamma_0>0$ to be determined later.

Our proof strategy largely parallels the line of reasoning established in the proof of Proposition \ref{nonlctrl}. Specifically, we use the favorable frequency gain available in the waveguide geometry $\mathbb{R}\times \mathbb{T}^2$ to handle the nonlinear terms without isolating the zero mode for image rotation.

We first use the nested frequency decomposition as Deng--Germain did in \cite{deng2019growth2}:
\begin{equation}\label{eq:decomp}
u_1 := P_{\le K/10}u,\qquad u_2 := u-u_1,
\qquad u_3 := P_{\le K^{\gamma_0}}u,\qquad u_4 := u_1-u_3.
\end{equation}
The nonlinear term is decomposed as follows:
$$
\left(i \partial_t+\Delta\right) P_K \mathcal{D} u=\mathcal{N}_1+\mathcal{N}_2+\mathcal{N}_3+\mathcal{N}_4
$$
where
\begin{equation}
\begin{aligned}\nonumber
\mathcal{N}_1= & P_K \mathcal{D}\left(\left|u_1\right|^{\mu-1} u_1\right), \\
\mathcal{N}_2= & P_K \mathcal{D}\left(\overline{u_2} \cdot F_{\mu-1}\left(u_1\right)\right), \\
\mathcal{N}_3= & \frac{\mu+1}{2} P_K \mathcal{D}\left(u_2 \left|u_3\right|^{\mu-1}\right), \\
\mathcal{N}_4= & P_K \mathcal{D}\left\{u_2 u_4 \int_0^1 F_{\mu-2}\left(u_3+\zeta u_4\right) \mathrm{d} \zeta+u_2 \overline{u_4} \int_0^1 F_{\mu-2}\left(u_3+\zeta u_4\right) \mathrm{d} \zeta\right\} \\
& +P_K \mathcal{D}\left\{\int_0^1\left[F_{\mu-2}\left(u_1+\zeta u_2\right) u_2^2+2 F_{\mu-2}\left(u_1+\zeta u_2\right) u_2 \overline{u_2}+F_{\mu-2}\left(u_1+\zeta u_2\right)\left(\overline{u_2}\right)^2\right] \zeta \mathrm{d} \zeta\right\} .
\end{aligned}
\end{equation}
Duhamel's formula toget
her with Strichartz estimates gives
\begin{align*}
    K^{-\sigma}\|P_K\mathcal{D}u\|_{L^{q_0}([0,\epsilon])}\leq \|\mathcal{D}u(0)\|_{H^1}& +\|\mathcal{N}_1\|_{X^{1,b-1}}+\|\mathcal{N}_2\|_{X^{1,b-1}}+\|\mathcal{N}_4\|_{X^{1,b-1}}\\
&+K^{-\sigma}\left\|\int_0^t e^{i(t-t')\Delta}\mathcal{N}_3(t')\mathrm{d}t'\right\|_{L^{q_0}([0,\epsilon])},
\end{align*}
for $0\leq t\leq \epsilon$.
All terms above, except for \(\mathcal{N}_3\), coincide with those in the proof of Proposition 5.1 of \cite{deng2019growth2}, and the estimates of their contributions hold similarly. Therefore, it suffices to establish the estimate for the contribution of \(\mathcal{N}_3\).

Let
\[\Omega(t,x):=\frac{\mu+1}{2}\left|u_3\right|^{\mu-1}.\]
Since
\begin{equation}
\begin{aligned}
\left\|\chi(t) P_K \mathcal{D}\left(u_2 \cdot P_{>K^{\gamma_0}} \Omega\right)\right\|_{X^{1, b-1}} &\lesssim K\left\|P_K \mathcal{D}\left(u_2 \cdot P_{>K^{\gamma_0}} \Omega\right)\right\|_{L_{t, x}^{10 / 7+}}\\
&\lesssim\left\|K \mathcal{D} u_2\right\|_{L_{t, x}^{10 / 3}}\left\|\mathcal{D} P_{>K^{\gamma_0}} \Omega\right\|_{L_{t, x}^{5 / 2+}}\\
&\lesssim \sum_{M \geq K^{\gamma_0}} M^{o(1)} \sum_{L \leq K^{\gamma_0}} M^{-1} L\left\|\mathcal{D} P_L u_3\right\|_{L_{t, x}^{q_0}}\\ &\lesssim \sum_{M \geq K^{\gamma_0}} M^{\sigma+} \lesssim K^{\sigma \gamma_0+},
\end{aligned}
\end{equation}
we may replace $\Omega$ by $P_{\leq K^{\gamma_0}}\Omega$ (and $u_2$ by $P_{[K / 4,4 K]} u$) in $\mathcal{N}_3$.

By inserting a suitable time cutoff $\chi(t)$, we can write 
\[\chi(t)P_{\leq K^{\gamma_0}}\Omega(t,x)=\chi(t) \int_{\xi\in \mathbb{R}\times\mathbb{Z}^2,|\xi|\leq K^{\gamma_0}}\int_{\mathbb{R}}d(\tau,\xi)e^{i( x\cdot\xi+t\tau)}\,\mathrm{d}\tau \mathrm{d}\xi .\]
Using equation \eqref{nls}, we compute for any interval $I$ with $|I|\lesssim \epsilon$:
\[\begin{split}\|\partial_t P_{\leq K^{\gamma_0}}\Omega\|_{L_{t,x}^2(I)}&\lesssim\|P_{\leq K^{\gamma_0}}u\|_{L_{t,x}^{\infty}}^{\mu-2}\cdot\bigg(K^{2\gamma_0}\|P_{\leq K^{\gamma_0}}u\|_{L_{t,x}^2(I)}+\|P_{\leq K^{\gamma_0}}(|u|^{\mu-1}u)\|_{L_{t,x}^2(I)}\bigg)\\
&\lesssim K^{\frac{(\mu-2)\gamma_0}{2}}(K^{2\gamma_0}+K^{3\gamma_0/2}\|u\|_{L_t^\infty L_{x}^\mu(I)}^\mu)\lesssim K^{\frac{(\mu+2)\gamma_0}{2}},\end{split}\]
where the last inequality follows from Sobolev embedding. This implies $\| P_{\leq K^{\gamma_0}}\Omega\|_{H_t^1H_{x}^2}\lesssim K^{\frac{(\mu+6)\gamma_0}{2}}$. Therefore,
\begin{equation}\label{verylow}
\int_{\xi\in \mathbb{R}\times\mathbb{Z}^2,|\xi|\leq K^{\gamma_0}}\int_{\mathbb{R}} \langle \tau\rangle^{1/3}|d(\tau,\xi)|\,\mathrm{d}\tau \mathrm{d}\xi   \lesssim K^{\frac{(\mu+6)\gamma_0}{2}}
\end{equation}
and
\begin{equation}
\int_{\xi\in \mathbb{R}\times\mathbb{Z}^2,|\xi|\leq K^{-20\gamma_0}}\int_{\mathbb{R}} \langle \tau\rangle^{1/3}|d(\tau,\xi)|\,\mathrm{d}\tau \mathrm{d}\xi   \lesssim K^{\frac{(\mu-14)\gamma_0}{2}}.
\end{equation}
This low-frequency estimate mirrors the analysis in the proof of Proposition \ref{nonlctrl}. Just as in the quintic case, the presence of the real line direction $\mathbb{R}$ in the frequency space $\mathbb{R}\times \mathbb{Z}^2$ ensures that the measure of the low-frequency region is sufficiently small to yield the necessary gain—a property absent in the pure torus setting.

Consequently, we have
\begin{align*}
|\int_0^t e^{i(t-t')\Delta}\mathcal{D}(P_{[K / 4,4 K]} u \cdot P_{\leq K^{\gamma_0}}\Omega)(t')\mathrm{d}t'| \lea &  \int_{|\xi|\leq K^{-20\gamma_0}}   \int_{\mathbb{R}}\langle \tau \rangle^{1/3} |d(\tau,\xi)| |h_{\tau,\xi}(t,x)|\,\mathrm{d}\tau \mathrm{d}\xi \\ &+ \int_{K^{-20\gamma_0}<|\xi|\leq K^{\gamma_0}}\int_{\mathbb{R}}\langle \tau \rangle^{1/3} |d(\tau,\xi)| |h_{\tau,\xi}(t,x)|\,\mathrm{d}\tau \mathrm{d}\xi      
\end{align*}
where $$h_{\tau,\xi}(t,x)=\langle \tau\rangle^{-1/3}\chi(t) \int_0^t e^{i(t-t')\Delta}\mathcal{D}\big(e^{i( x\cdot\xi+ t'\tau)}P_{[K / 4,4 K]}u(t',x)\big)dt'\,.$$
Since \eqref{strxsb} and  $\langle \tau\rangle^{1/3}\|h_{\tau,\xi}\|_{X^{1,b}}\lesssim 1$, we obtain:
\begin{equation}
    \begin{aligned}
K^{-\sigma}\left\|\int_0^t e^{i(t-t')\Delta}\mathcal{D}(P_{[K / 4,4 K]} u \cdot P_{\leq K^{\gamma_0}}\Omega)(t')\mathrm{d}t'\right\|_{L^{q_0}} &\lesssim (K^{\frac{(\mu+6)\gamma_0}{2}}+K^{\frac{(\mu-14)\gamma_0}{2}})\|h_{\tau,\xi}\|_{X^{1,b}}\\
&\lesssim K^{\frac{(\mu-14)\gamma_0}{2}},
\end{aligned}
\end{equation}
if $|\tau|\gtrsim K^{30\gamma_0}$.

For the case $|\tau|\lea K^{30\gamma_0}$ and $K^{-20\gamma_0}<|\xi|\lesssim K^{\gamma_0}$ (the case $|\xi|\leq K^{-20\gamma_0}$ is already subsumed by the above estimate), we utilize the resonance set decomposition $P_{[K/4,4K]}u=P_{\mathcal{R}}u+P_{[K/4,4K]}P_{\mathcal{R}'}u$, where
\[\mathcal{R}=\big\{\w{\xi}\in \bR\times\mathbb{Z}^2:|\w{\xi}|\sim K,\l| |\w{\xi}+\xi|^2-|\w{\xi}|^2 \r|\lesssim K^{45\gamma_0}\big\}\] 
and $\mathcal{R}'=\bR\times\mathbb{Z}^2-\mathcal{R}$.  Following the resonance analysis utilized in the previous proposition, we find:
$$K^{-\sigma}\left\|\int_0^t e^{i(t-t')\Delta}\mathcal{D}(P_{\mathcal{R}}u \cdot P_{\leq K^{\gamma_0}}\Omega)(t')\mathrm{d}t'\right\|_{L^{q_0}} \lesssim K^{\frac{(\mu+6)\gamma_0}{2}}\left( \frac{K^{2+65\gamma_0}}{K^3} \right)^{(\frac{1}{2}-\frac{5}{3q_0})-}\lesssim K^{\frac{(68\mu-177)\gamma_0-(\mu-3)}{6}+},$$
with  \(\gamma_0\) sufficiently small, the exponent becomes negative. The contribution of \(P_{K_1}P_{\mathcal{R}'}u\) is bounded by \(K^{-\gamma_0/2}\) as before.

Finally,  we arrive at
\[
\left\| v(t) - \mathrm{e}^{it\Delta} v(0) \right\|_{L^{q_0}([0,\epsilon])} \lesssim K^{\sigma+}\cdot\max\{K^{\sigma\gamma_0},K^{\frac{(68\mu-177)\gamma_0-(\mu-3)}{6}},K^{-\gamma_0/2}\}.
\] Choosing $\gamma_0 = \frac{\mu-3}{65\mu-162}$ yields that $\left\| v(t) - \mathrm{e}^{it\Delta} v(0) \right\|_{L^{q_0}([0,\epsilon])} \lesssim K^{\sigma+\theta_1+}$ with $
\theta_1=\frac{(\mu-3)(\mu-5)}{2(65 \mu-162)}<0
$.
By time translation, one also gets that
\[
\left\| v(t) - \mathrm{e}^{i(t-m\epsilon)\Delta} v(m\epsilon) \right\|_{L_{t,x}^{q_0}([m\epsilon,(m+1)\epsilon] \times (\mathbb{R}\times \mathbb{T}^2))} \lesssim K^{\sigma+\theta_1+}.
\]
for each $0 \leq m \leq K^{\gamma}$. Proceeding with the same argument as in \cite{deng2019growth2}, we obtain the desired result.
\end{proof}

By combining Proposition \ref{nonlctrl} with the Proposition \ref{alm-con-ene} and Proposition \ref{conversation2} with Proposition \ref{almostconse2}, we apply the standard argument to be able to prove Theorem \ref{thm:growth of solutions to NLS on 3D waveguide}. We omit the detailed derivations herein; one may refer to \cite{deng2019growth,deng2019growth2} for complete details.

\bibliographystyle{amsplain}
\bibliographystyle{plain}

\begin{thebibliography}{10}

\bibitem{Barron2021}
Alex Barron, \emph{On global-in-time {S}trichartz estimates for the
  semiperiodic {S}chr\"{o}dinger equation}, Anal. PDE \textbf{14} (2021),
  no.~4, 1125--1152.

\bibitem{BCP2021}
Alexander Barron, Michael Christ, and Benoit Pausader, \emph{Global endpoint
  {S}trichartz estimates for {S}chr\"odinger equations on the cylinder {$\Bbb
  R\times \Bbb T$}}, Nonlinear Anal. \textbf{206} (2021), Paper No. 112172, 7.

\bibitem{Bourgain1993NLS}
Jean Bourgain, \emph{Fourier transform restriction phenomena for certain
  lattice subsets and applications to nonlinear evolution equations. {I}.
  {S}chr\"odinger equations}, Geom. Funct. Anal. \textbf{3} (1993), no.~2,
  107--156.

\bibitem{bourgain1996growth}
\bysame, \emph{On the growth in time of higher {S}obolev norms of smooth
  solutions of {H}amiltonian {PDE}}, Int. Math. Res. Not. IMRN (1996), no.~6,
  277--304.

\bibitem{BD2015}
Jean Bourgain and Ciprian Demeter, \emph{The proof of the {$l^2$} decoupling
  conjecture}, Ann. of Math. \textbf{182} (2015), no.~1, 351--389.

\bibitem{I-team1}
James Colliander, Markus Keel, Gigiola Staffilani, Hideo Takaoka, and Terence
  Tao, \emph{Global well-posedness for {S}chr\"{o}dinger equations with
  derivative}, SIAM J. Math. Anal. \textbf{33} (2001), no.~3, 649--669.

\bibitem{colliander2002polynomial}
\bysame, \emph{Polynomial upper bounds for the orbital instability of the 1{D}
  cubic {NLS} below the energy norm}, Discrete Contin. Dyn. Syst. \textbf{9}
  (2003), no.~1, 31--54.

\bibitem{DPST}
Daniela De~Silva, Nata\v{s}a Pavlovi\'{c}, Gigliola Staffilani, and Nikolaos
  Tzirakis, \emph{Global well-posedness for a periodic nonlinear
  {S}chr\"{o}dinger equation in 1{D} and 2{D}}, Discrete Contin. Dyn. Syst.
  \textbf{19} (2007), no.~1, 37--65. \MR{2318273}

\bibitem{deng2024growth}
Mingming Deng and Kailong Yang, \emph{On the growth of high {S}obolev norms of
  the cubic nonlinear {S}chr\"odinger equation on {$\Bbb R\times\Bbb T$}},
  Differential Integral Equations \textbf{37} (2024), no.~5-6, 337--358.

\bibitem{DFZ2025}
Yangkendi Deng, Chenjie Fan, and Zehua Zhao, \emph{Sharp {$L^4$} {S}trichartz
  estimate for {H}yperbolic {S}chr\"odinger equation on $\mathbb{R}\times
  \mathbb{T}$}, arXiv preprint arXiv:2511.15157 (2025).

\bibitem{deng2019growth}
Yu~Deng, \emph{On growth of {S}obolev norms for energy critical {NLS} on
  irrational tori: small energy case}, Comm. Pure Appl. Math. \textbf{72}
  (2019), no.~4, 801--834.

\bibitem{deng2019growth2}
Yu~Deng and Pierre Germain, \emph{Growth of solutions to {NLS} on irrational
  tori}, Int. Math. Res. Not. IMRN (2019), no.~9, 2919--2950.

\bibitem{DGG2017}
Yu~Deng, Pierre Germain, and Larry Guth, \emph{Strichartz estimates for the
  {S}chr\"odinger equation on irrational tori}, J. Funct. Anal. \textbf{273}
  (2017), no.~9, 2846--2869.

\bibitem{Dodson3}
Benjamin Dodson, \emph{Global well-posedness and scattering for the defocusing,
  {$L^{2}$}-critical nonlinear {S}chr\"{o}dinger equation when {$d\geq3$}}, J.
  Amer. Math. Soc. \textbf{25} (2012), no.~2, 429--463.

\bibitem{HP}
Zaher Hani and Benoit Pausader, \emph{On scattering for the quintic defocusing
  nonlinear {S}chr\"{o}dinger equation on {$\Bbb R\times\Bbb T^2$}}, Comm. Pure
  Appl. Math. \textbf{67} (2014), no.~9, 1466--1542.

\bibitem{MR3406826}
Zaher Hani, Benoit Pausader, Nikolay Tzvetkov, and Nicola Visciglia,
  \emph{Modified scattering for the cubic {S}chr\"{o}dinger equation on product
  spaces and applications}, Forum Math. Pi \textbf{3} (2015), e4, 63.

\bibitem{HK2024}
Sebastian Herr and Beomjong Kwak, \emph{{S}trichartz estimates and global
  well-posedness of the cubic {NLS} on {$\Bbb {T}^2$}}, Forum Math. Pi
  \textbf{12} (2024), Paper No. e14, 21.

\bibitem{herr2025global}
\bysame, \emph{Global well-posedness of the cubic nonlinear {S}chr\"odinger
  equation on $\mathbb{T}^{2}$}, arXiv preprint arXiv:2502.17073 (2025).

\bibitem{HTT1}
Sebastian Herr, Daniel Tataru, and Nikolay Tzvetkov, \emph{Global
  well-posedness of the energy-critical nonlinear {S}chr\"odinger equation with
  small initial data in {$H^1(\Bbb T^3)$}}, Duke Math. J. \textbf{159} (2011),
  no.~2, 329--349.

\bibitem{IPT3}
Alexandru~D. Ionescu and Benoit Pausader, \emph{The energy-critical defocusing
  {NLS} on {$\Bbb T^3$}}, Duke Math. J. \textbf{161} (2012), no.~8, 1581--1612.

\bibitem{KV2016}
Rowan Killip and Monica Vi\c{s}an, \emph{Scale invariant {S}trichartz estimates
  on tori and applications}, Math. Res. Lett. \textbf{23} (2016), no.~2,
  445--472.

\bibitem{kwak2024global}
Beomjong Kwak, \emph{Global well-posedness of the energy-critical nonlinear
  {S}chr\" odinger equations on $\mathbb{T}^d$}, arXiv preprint
  arXiv:2411.18163 (2024).

\bibitem{kwak2024critical}
Beomjong Kwak and Soonsik Kwon, \emph{Critical local well-posedness of the
  nonlinear {S}chr{\"o}dinger equation on the torus}, Annales de l'Institut
  Henri Poincar{\'e} C (2024).

\bibitem{LZ2025}
Baoping Liu and Xu~Zheng, \emph{On sharp {S}trichartz estimate for hyperbolic
  {S}chr\"{o}dinger equation on $\mathbb{T}^3$}, arXiv preprint
  arXiv:2510.01886 (2025).

\bibitem{planchon2017growth}
Fabrice Planchon, Nikolay Tzvetkov, and Nicola Visciglia, \emph{On the growth
  of {S}obolev norms for {NLS} on 2-and 3-dimensional manifolds}, Analysis \&
  PDE \textbf{10} (2017), no.~5, 1123--1147.

\bibitem{S}
Thomas Schneider, \emph{Nonlinear optics in telecommunications}, Springer
  Science \& Business Media, 2013.

\bibitem{SL}
Allan~W Snyder, John~D Love, et~al., \emph{Optical waveguide theory}, vol. 175,
  Chapman and hall London, 1983.

\bibitem{sohinger2011bounds1}
Vedran Sohinger, \emph{Bounds on the growth of high {S}obolev norms of
  solutions to nonlinear {S}chr\"odinger equations on {$\Bbb R$}}, Indiana
  Univ. Math. J. \textbf{60} (2011), no.~5, 1487--1516.

\bibitem{sohinger2011bounds}
\bysame, \emph{Bounds on the growth of high {S}obolev norms of solutions to
  nonlinear {S}chr\"odinger equations on {$S^1$}}, Differential Integral
  Equations \textbf{24} (2011), no.~7-8, 653--718.

\bibitem{staffilani1997growth}
Gigliola Staffilani, \emph{On the growth of high {S}obolev norms of solutions
  for {K}d{V} and {S}chr{\"o}dinger equations}, Duke Math. J. \textbf{86}
  (1997), 109--142.

\bibitem{staffilani1997quadratic}
\bysame, \emph{Quadratic forms for a 2-{D} semilinear {S}chr{\"o}dinger
  equation}, Duke Math. J. \textbf{86} (1997), 79--107.

\bibitem{Strichartz1977}
Robert~S. Strichartz, \emph{Restrictions of {F}ourier transforms to quadratic
  surfaces and decay of solutions of wave equations}, Duke Math. J. \textbf{44}
  (1977), no.~3, 705--714.

\bibitem{takaoka2024growth}
Hideo Takaoka, \emph{On the growth of {S}obolev norm for the cubic {NLS} on two
  dimensional product space}, J. Differential Equations \textbf{394} (2024),
  296--319.

\bibitem{TT2001}
Hideo Takaoka and Nikolay Tzvetkov, \emph{On 2{D} nonlinear {S}chr\"{o}dinger
  equations with data on {${\Bbb R}\times\Bbb T$}}, J. Funct. Anal.
  \textbf{182} (2001), no.~2, 427--442.

\bibitem{Taobook}
Terence Tao, \emph{Nonlinear dispersive equations}, CBMS Regional Conference
  Series in Mathematics, vol. 106, Published for the Conference Board of the
  Mathematical Sciences, Washington, DC; by the American Mathematical Society,
  Providence, RI, 2006, Local and global analysis.

\bibitem{tzvetkov2016well}
Nikolay Tzvetkov and Nicola Visciglia, \emph{Well-posedness and scattering for
  nonlinear {S}chr\"odinger equations on {$\Bbb{R}^d\times\Bbb{T}$} in the
  energy space}, Rev. Mat. Iberoam. \textbf{32} (2016), no.~4, 1163--1188.

\end{thebibliography}

\end{document}